\documentclass[10pt]{amsart}
\usepackage{amsmath, amssymb, amsthm, bm, mathtools, enumitem, caption}
\usepackage{hyperref}

\usepackage[noabbrev,capitalize]{cleveref}
\usepackage{adjustbox}
\crefname{equation}{}{}
\usepackage{graphicx, tikz}
\usepackage[textheight=8.8in, textwidth=6.8in]{geometry}
\usepackage{comment}

\usepackage{microtype}

\newtheorem{theorem}{Theorem}[section]
\newtheorem{lemma}[theorem]{Lemma}
\newtheorem{corollary}[theorem]{Corollary}
\newtheorem{proposition}[theorem]{Proposition}

\newtheorem*{conjecture*}{Conjecture}

\theoremstyle{definition}

\newtheorem*{definition}{Definition}

\theoremstyle{remark}
\newtheorem*{remark}{Remark}

\numberwithin{equation}{section}

\newcommand{\R}{\mathbb R}
\newcommand{\N}{\mathbb N}

\DeclareMathOperator{\Tr}{Tr}
\DeclareMathOperator{\GL}{GL}

\DeclareMathOperator{\sgn}{sgn}

\newcommand{\leg}[2]{\genfrac{(}{)}{}{}{#1}{#2}}

\DeclareMathOperator{\sign}{sign}
\DeclareMathOperator{\spt}{spt}

\newcommand{\lam}{\lambda}

\newcommand{\HH}{\mathbb H}

\newcommand{\C}{\mathbb C}
\newcommand{\SL}{\mathrm{SL}}

\newcommand{\Q}{\mathbb Q}
\newcommand{\RR}{\mathbb R}
\newcommand{\Z}{\mathbb Z}

\title[Euler-type recurrence relations for partition functions with congruence conditions]{Euler-type recurrence relations for partition functions with congruence conditions}

\author{Wissam Raji and Hasan Saad}

\address{Dept. of Mathematics, American University of Beirut, Beirut, Lebanon}
\email{wr07@aub.edu.lb}

\address{Dept. of Mathematics, Louisiana State University, Baton Rouge, LA 70803, United States of America}
\email{hsaad@lsu.edu}

\begin{document}
	
	\begin{abstract}
    We study partition functions $p_{\delta,g}(n)$ counting partitions into parts congruent to $0$ or $\pm g \pmod\delta$. Using generalized Dedekind eta functions and Rankin–Cohen brackets, we derive infinite families of Euler-type recurrences involving divisor sums and Fourier coefficients of cusp forms. We also obtain an explicit recurrence for $\delta=5$, which as a corollary, gives a Ramanujan-type congruence. As a corollary of our method of proof, we obtain a Rademacher-type formula involving Kloosterman sums and Bessel functions.
    \end{abstract}
    \maketitle
	\section{Introduction and Statement of Results}
	
	A partition of $n$ is a nonincreasing sequence of positive integers that sums to $n.$ The partition numbers $p(n)$ denote the number of partitions of $n$ and satisfy the famous recurrence relation of Euler (see page 12 of \cite{Andrews})
	\begin{equation}\label{eq:EulerRecurrence}
	p(n)=p(n-1)+p(n-2)-p(n-5)-p(n-7)+\ldots=\sum\limits_{j\in\Z\setminus\{0\}}(-1)^{j+1}p(n-\omega(j)),
	\end{equation}
	where $\omega(j):=\frac{3j^2+j}{2}$ is the $j$-th pentagonal number.
	
	Recently, Gomez, Ono, the second author, and Singh proved \cite{Gomez} that \refeq{eq:EulerRecurrence} is a special case of an infinite family of pentagonal number recurrence relations for $p(n)$ which involve modular forms of various weights on $\SL_2(\Z).$ Their work was extended \cite{Bhowmik} by Bhowmik, Tsai, and Ye to $t$-color and $t$-regular partition numbers for $t\in\{2,3\}$ and was modified \cite{Wang} by Wang to obtain Euler-like recursions for the smallest parts function $\spt(n).$ 
	
	Here we consider the analogous situation for partitions with certain congruence conditions imposed on their parts. To make this precise, let $\delta\geq 5$ and $g\in\{1,\ldots,\delta-1\}$ with $g\neq\frac{\delta}{2}.$ If $n\geq 1,$ then we define the partition function $p_{\delta,g}$ to count the number of partitions whose parts are congruent to $0$ or $\pm g$ modulo $\delta,$ that is,
	$$
	p_{\delta,g}(n):=\#\{\pi\vdash n: \pi = n_1\cdot m_1+\ldots+n_k\cdot m_k, m_i\equiv 0\text{ or }m_i\equiv\pm g\pmod\delta\text{ for }1\leq i\leq k\}
	$$
	with the convention that $p_{\delta,g}(0):=1.$ By applying the Jacobi triple product identity (Theorem 2.8 of \cite{Andrews})
	$$
	\prod\limits_{m=1}^\infty (1-x^{2m})(1+x^{2m-1}y^2)\left(1+\frac{x^{2m-1}}{y^2}\right)=\sum\limits_{n\in\Z} x^{n^2}y^{2n}
	$$
	with $x=q^{\frac{1}{2}}$ and $y^2=-q^{-\frac{1}{2}}\cdot q^{\frac{g}{\delta}},$ we have that
	\begin{equation}\label{eq:gEtaExp}
	\prod\limits_{\substack{m\geq 1 \\ m\equiv 0,\pm g\pmod\delta}} (1-q^{\frac{m}{\delta}})=\sum\limits_{n\in\Z}(-1)^nq^{\frac{n^2-n}{2}+\frac{ng}{\delta}}.
	\end{equation}
	On the other hand, it is clear that
	$$
	\frac{1}{\prod\limits_{\substack{m\geq 1 \\ m\equiv 0,\pm g\pmod\delta}}(1-q^{\frac{m}{\delta}})} = \sum\limits_{n\geq 0} p_{\delta,g}(n)q^{\frac{n}{\delta}}.
	$$
	By multiplying the above expressions and comparing coefficients of $q$ on both sides, if $n\geq 1,$ then we have that
	\begin{equation}\label{eq:recRel}
	p_{\delta,g}(n)=\sum\limits_{k\in\Z\setminus\{0\}}(-1)^{k+1}p_{\delta,g}(n-\omega(\delta,g;k)),
	\end{equation}
	where $\omega(\delta,g;k):=\frac{\delta(k^2-k)}{2}+kg.$ 
	
	The main goal of this paper is to show that \eqref{eq:recRel} is a special case of an infinite family of recurrence relations for the quantities $p_{\delta,g}$ which arise from Fourier coefficients of modular forms. 
	
	As an example of our results, we prove the following theorem, which gives a recurrence relation for $p_{5,1}$ in terms of divisor sum functions and Fourier coefficients of the unique cusp form of weight $4$ and level $5.$ 
	
	\begin{theorem}\label{thm:ExRecurDelta5g1}
		If $n\geq 1,$ then we have
			\begin{align*}
			p_{5,1}(n) &= \frac{80}{3(27+20n(-9+10n))}\Bigg[\sum\limits_{k\in\Z\setminus\{0\}}(-1)^{k+1}P(n,k)p_{5,1}(n-\omega(5,1;k)) \\
			&+\frac{73}{26}\sum\limits_{d\mid n}d^3+\frac{3125}{26}\sum\limits_{d\mid \frac{n}{5}}d^3+\frac{1}{2}\sum\limits_{d\mid n}\leg{d}{5}d^3-\frac{30}{13}b(n)\Bigg],
			\end{align*}
			where $P(n,k)=\frac{1}{80}\left(((3-10k)^4-60(3-10k)^2n+600n^2)\right), \leg{\cdot}{5}$ is the Legendre character modulo $5,$ and where the coefficients $b(n)$ are given by $$
			q\prod\limits_{n=1}^\infty(1-q^n)^4(1-q^{5n})^4=:\sum\limits_{n=1}^{\infty}b(n)q^n.
			$$
	\end{theorem}
	
	\begin{remark}
		As a corollary of Theorem~\ref{thm:ExRecurDelta5g1}, we easily obtain the Ramanujan-type congruence
		$$
		b(n) \equiv \sum\limits_{d\mid n}d^3-\sum\limits_{d\mid\frac{n}{5}}d^3\pmod{13}
		$$
		for all $n\geq 1.$
	\end{remark}
	
	Theorem~\ref{thm:ExRecurDelta5g1} is the case $\nu=2, \delta=5$ and $g=1.$ In general, for $\nu\geq 1,$ the recurrence formula is given in terms of Hecke traces
	of twisted Dirichlet $L$-series over the space of cusp forms of weight $2\nu$ on the group
	$$
	\Gamma^1(\delta):=\left\{\begin{pmatrix}
		a & b \\
		c & d
	\end{pmatrix}\in\SL_2(\Z) : b\equiv 0\mod\delta, a\equiv d\equiv 1\mod\delta\right\}.
	$$ In order to make this precise, we introduce some notation. If $x\in\R,$ we define
	$$
	P_1(x):=\begin{cases}
		\{x\}-\frac{1}{2} & \text{ if }x\not\in\Z \\
		0 & \text{ if }x\in\Z,
	\end{cases}\qquad\text{ and }\qquad P_2(x)=\{x\}^2-\{x\}+\frac{1}{6},
	$$
	where $\{x\}=x-\lfloor x\rfloor$ is the fractional part of $x.$ 
	In this notation, if $\delta\geq 5$ and $0<g<\delta$ with $g\neq\frac{\delta}{2}$ and $A=\begin{pmatrix}
		a & b \\
		c & d
	\end{pmatrix}\in\SL_2(\Z),$ we write
	$$
	\mu_{\delta,g}(A)=\frac{a}{c}P_2\left(\frac{g}{\delta}\right)+\frac{d}{c}P_2\left(\frac{ag}{\delta}\right)-2s\left(a,\frac{c}{\delta};0,\frac{g}{\delta}\right),
	$$
	where $s(h,k;x,y)$ is the Meyer sum
	$$
	s(h,k;x,y)=\sum\limits_{\mu\mod k}P_1\left(h\left(\frac{\mu + y}{k}\right)+x\right)\cdot P_1\left(\frac{\mu + y}{k}\right).
	$$
	Using this, we define a multiplier system $\varepsilon_{\delta,g}$ on $\SL_2(\Z)$ by
	$$
	\varepsilon_{\delta,g}(A)=\varepsilon_\eta(A)\cdot e^{\pi i\mu_{\delta,g}(A)},
	$$
	where $\varepsilon_\eta$ is the multiplier system for the Dedekind eta function (see p. 140 of \cite{HMFBook} for an explicit description of $\varepsilon_\eta$). 
	
	Using this multiplier system, we associate certain arithmetic quantities to the cusps of $\pm\Gamma^1(\delta).$ If $\delta\geq 5, m\geq 1,$ and $\rho=\frac{-h_\rho}{g_\rho}$ is a cusp of $\pm\Gamma^1(\delta),$ we denote (see Lemma~\ref{lem:cuspParamLemma}) by $t_\rho=\frac{\delta}{(h_\rho,\delta)}$ the cusp width of $\rho,$ and we write $K_\rho=\frac{t_\rho P_2\left(\frac{h_\rho g}{\delta}\right)}{2}+\frac{t_\rho}{24}$ and denote by  $\kappa_\rho=\{K_\rho\}$ the cusp parameter of $\rho$ with respect to the multiplier system $\varepsilon_{\delta,g}$ (see Section 3.3 of \cite{RankinBook}). Moreover, we define $\widetilde{\kappa_\rho}=\{-\kappa_\rho\}$ and $\beta_{\rho,m}:=\frac{t_\rho}{m-\widetilde{\kappa_\rho}}.$ To each cusp $\rho,$ we associate a matrix $L_\rho=\begin{pmatrix}
		e_\rho & f_\rho \\
		g_\rho & h_\rho
	\end{pmatrix}\in\SL_2(\Z)$ with $L_\rho^{-1}\infty=\rho.$ If $\rho$ is a cusp and $0\leq l<\frac{\delta P_2\left(\frac{h_\rho g}{\delta}\right)}{2}+\frac{\delta}{24},$ we define
	$$
	pp_\rho(l):=\begin{cases}
		\zeta_{\delta}^{-\frac{lf_\rho g}{\alpha_\rho}}\delta_{\alpha_\rho\mid l}+\zeta_{\delta}^{\frac{lf_\rho g}{\beta_\rho}}\delta_{\beta_\rho\mid l} & \text{ if }\alpha\neq 0\text{ and }l\neq 0 \\
		1 & \text{ if }l = 0 \\
		0 & \text{ otherwise} 
	\end{cases},
	$$
	where $\zeta_\delta:=e^{\frac{2\pi i}{\delta}}$ and $\alpha_\rho,\beta_\rho\in\{0,1,\ldots,\delta-1\}$ are the representatives of $h_\rho g,-h_\rho g$ in $\Z/\delta\Z.$  
	
	In this notation, if $\nu\geq 1, \delta\geq 5, 0<g<\delta$ with $g\neq\frac{\delta}{2}, m\geq 1, s\in\C$ with $\Re(s)>\nu+\frac{1}{2},$ and  $f$ is a cusp form of weight $2\nu$ on $\Gamma^1(\delta)$ with Fourier expansion at a cusp $\rho$ given by
	$$
	(f|_{2\nu} L_\rho^{-1}) = \sum\limits_{j=1}^\infty a_{f,\rho}(j)q^{\frac{j}{\delta}},
	$$ 
	we define the twisted Dirichlet $L$-series
	$$
	D(f;\delta,g;\rho,m;s):=\sum\limits_{\substack{n \in \Z \\ 
	\delta -2h_\rho g-2\delta n\neq 0}}(-1)^n\zeta_{\delta}^{-nf_\rho g}\overline{a_{f,\rho}\left(\delta\left(\frac{-m+\widetilde{\kappa_\rho}}{t_\rho}+
		\frac{(\delta-2h_\rho g-2\delta n)^2}{8\delta^2}\right)\right)}\cdot\frac{\sgn(\delta-2h_\rho g-2\delta n)}{|\delta-2h_\rho g-2\delta n|^{s}},
	$$
	where $\sgn(x)=\begin{cases}
		1 & x>0 \\
		-1 & x<0
	\end{cases}.$
	The recurrence relations we obtain will be in terms of infinite sums of values of the Dirichlet series $D(f;\delta,g;\rho,m;s).$ In the notation above, we define
	$$
	D_{f,\delta,g,\rho,m}:=\sum\limits_{i=0}^{\nu-2}\sum\limits_{j=0}^\infty c(\rho,m,\nu,i,j)\cdot D(f;\delta,g;\rho,m;2\nu+2i+2j+1),
	$$
	where
	$$
	c(\rho,m,\nu,i,j):= \frac{(-1)^i(\nu-i-1)^{\overline{\nu}}\left(\frac{3}{2}\right)^{\overline{i}}}{\left(-\frac{1}{2}-i\right)^{\overline{\nu}}\left(\frac{5}{2}\right)^{\overline{i}}}\cdot\frac{(2\nu+j-2)!}{i!j!(2\nu-i-2)!}\cdot\left(\frac{8\delta^2}{\beta_{\rho,m}}\right)^{i+j}
	$$
	with $(a)^{\overline{n}}$ denoting the rising factorial. Finally, if $\delta\geq 5, 0<g<\delta$ with $g\neq\frac{\delta}{2}$ and $h\in\Z,$ we define
	$$
	\alpha_0(\delta,g,h):=\begin{cases}
		(1-\zeta_\delta^{-h})e^{\pi iP_1\left(\frac{h}{\delta}\right)} & \text{ if }g\equiv 0\mod\delta\text{ and }h\not\equiv 0\mod\delta \\
		1 & \text{ otherwise}.
	\end{cases}
	$$
	In this notation, we have the associated trace 
	$$
	\Tr_{\delta,g;\nu}(n):=\sum\limits_f \sum\limits_{\substack{\rho\\
	0<m\leq \lceil K_\rho\rceil }}\left(\frac{8}{\beta_{\rho,m}}\right)^{\frac{3}{2}}pp_\rho\left(\frac{\delta}{t_\rho}(-m+\lceil K_\rho\rceil)\right)\alpha_0(-f_\rho g)\frac{D_{f,\delta,g,\rho,m}}{\langle f,f\rangle}\cdot a_f(n),
	$$
	where  $f(\tau):=q^{\frac{1}{\delta}}+ \sum\limits_{n\geq 2} a_f(n)q^{\frac{n}{\delta}}$ runs over an orthogonal basis of normalized Hecke eigenforms on $\Gamma^1(\delta)$ of weight $2\nu$ and where $\langle\cdot,\cdot\rangle$ is the Petersson inner product defined by
	$$
	\langle f,g\rangle = \iint_{\Gamma^1(\delta)\backslash \HH} f(x+iy)\overline{g}(x+iy)y^{2\nu-2}dxdy
	$$
	for all cusp forms $f,g$ of weight $2\nu$ on $\Gamma^1(\delta).$
	
	Moreover, the recurrence relations involve divisor sums and generalized Bernoulli numbers. Recall that if $\chi$ is a primitive Dirichlet character of conductor $N(\chi)$ and $k\geq 1,$ then we have the associated divisor sums
	$$
	\sigma_{k-1}(\chi,n):=\sum\limits_{d\mid n}\chi(d)d^{k-1}
	$$
	and the generalized Bernoulli numbers $B_{k,\chi}$ defined by
	$$
	\sum\limits_{a=1}^{N(\chi)}\chi(a)\frac{te^{at}}{e^{N(\chi)t}-1}=:\sum\limits_{n\geq 0}B_{n,\chi}\frac{t^n}{n!}.
	$$
	If $\nu\geq 2, \delta\geq 5$ and $0<g<\delta$ with $g\neq \frac{\delta}{2},$ then define the $(\lfloor\frac{\delta}{2}\rfloor + 1)\times (\lfloor\frac{\delta}{2}\rfloor + 1)$-matrix $M$ as $$
	M_{h,(\chi,e)}:=\frac{-B_{2\nu,\overline{\chi}}}{4\nu}\gcd(e,h)^{2\nu}\chi\left(\frac{h}{\gcd(e,h)}\right),
	$$ 
	where $\chi$ ranges over the even primitive characters modulo $N(\chi)\mid\delta, e$ ranges over positive divisors of $\frac{\delta}{N(\chi)},$ and $0\leq h\leq\left\lfloor\frac{\delta}{2}\right\rfloor.$ Similarly, define $(\lfloor\frac{\delta}{2}\rfloor + 1)\times 1$-matrix $N$ as
	$$
	N_{h,0}:=\frac{(2\nu-2)!}{\nu!(\nu-2)!}\cdot\left(\frac{P_2\left(\frac{hg}{\delta}\right)}{2}+\frac{1}{24}\right)^\nu.
	$$
	Denote by $\alpha_{\chi,e}$ the $(\chi,e)$-entry of the matrix $M^{-1}N.$ In this notation, we obtain a general recurrence relation.
	
	\begin{theorem}\label{thm:generalRecurrence}
		If $\delta\geq 5, 0<g<\delta$ with $g\neq\frac{\delta}{2},$ $\nu\geq 2,$ and $n\geq 1,$ then we have
		$$
		\sum\limits_{k\in\Z} g_\nu(\delta,g;n,k)(-1)^k p(\delta,g;n-\omega(\delta,g;k)) = \sum\limits_{\chi,e}\alpha_{\chi,e}\sigma_{2\nu-1}\left(\overline{\chi},n\cdot\frac{N(\chi)\cdot e}{\delta}\right) -\frac{\delta^{2\nu+2}\Gamma\left(\nu-\frac{1}{2}\right)\Gamma\left(\nu+\frac{1}{2}\right)}{3\pi^{2\nu}} \Tr_{\delta,g;\nu}(n),
		$$
		where\footnote{We define $\sigma_{k-1}(\overline{\chi},x)=0$ when $x\not\in\N.$} 
		$$
			g_\nu(\delta,g;n,k) = \sum\limits_{r=0}^\nu\frac{\Gamma\left(\nu+\frac{1}{2}\right)\Gamma\left(\nu-\frac{1}{2}\right)}{r!(\nu-r)!\Gamma\left(r-\frac{1}{2}\right)\Gamma\left(\nu-r+\frac{1}{2}\right)}\left(\frac{1}{24}+\frac{P_2(\frac{g}{\delta})}{2}+\frac{\omega(\delta,g;k)}{\delta}\right)^{\nu-r}\left(\frac{1}{24}+\frac{P_2(\frac{g}{\delta})}{2}+\frac{\omega(\delta,g;k)}{\delta}-\frac{n}{\delta}\right)^r.
		$$
	\end{theorem}
	
	\begin{remark}
		In order for Theorem~\ref{thm:generalRecurrence} to be a recurrence relation, we need to show that $g_\nu(\delta,g;n,0)\neq 0$ for all $n\geq 1.$ To see this, first note that by basic manipulation of the Gamma functions, we have that
		$$
		g_\nu(\delta,g;n,0)=\frac{(2\nu-1)(2\nu-2)_{\nu-2}^2}{2^{2\nu-2}(2\nu)!}h_\nu(k(\delta,g;n)),
		$$
		where $(a)_n:=a(a-1)\cdot\ldots\cdot(a-n+1)$ is the falling factorial, $h_\nu$ is a polynomial with integer coefficients and constant term $\pm1$ and where 
		$$
		k(\delta,g;n):=\delta^2+12g^2-12g\delta+2\delta^2-24n\delta.
		$$
		Now, note that if $k(\delta,g;n)=\pm1$ for some $n\in\Z,$ then taking both sides modulo $12,$ we would have that
		$$
		3\delta^2\equiv\pm1\mod 12,
		$$
		which is impossible. Therefore, if $n\geq 1,$ there exists some prime $\ell>1$ with $l\mid k(\delta,g;n),$ and thus, $h_\nu(k(\delta,g;n))\equiv\pm1\mod\ell.$ This gives our claim that $g_\nu(\delta,g;n,0)\neq 0$ for all $n\geq 1.$
	\end{remark}
	
	As a consequence of our method of proof of Theorem~\ref{thm:generalRecurrence}, we obtain a Rademacher-type formula for $p_{\delta,g}(n).$ To make this precise, we first introduce some notation. If $c>0$ is fixed, $\rho$ is a cusp, and $L_\rho=\begin{pmatrix}
		e_\rho & f_\rho \\
		g_\rho & h_\rho
	\end{pmatrix}\in\SL_2(\Z)$ with $L_\rho^{-1}\infty=\rho,$ denote by $d_+$ and $d_{-}$ the representatives in $[0,\delta-1]$ of $h_\rho$ and $-h_\rho$ modulo $\delta$ respectively and for any $d$ with $\gcd(c,d)=1,$ denote by $d^\ast$ the unique solution in $[0,c-1]$ of $dd^\ast\equiv 1\mod c$. For any $c>0,$  write
		\begin{align}\label{eq:FLCDef}
		\mathcal{F}_{L_\rho,\delta}(c):=\Bigg\{&\begin{pmatrix}
			(d_{\pm}+n\delta)^\ast + mc & \frac{	((d_{\pm}+n\delta)^\ast + mc)(d_{\pm}+n\delta) - 1}{c} \\
			c & d_{\pm}+n\delta
		\end{pmatrix} : 0\leq n < \frac{c\delta-d_\pm}{\delta}, \gcd(d_\pm+n\delta,c)=1, \\
		&0\leq m < \frac{ct_\rho-(d_{\pm}+n\delta)^\ast}{c}, \frac{((d_{\pm}+n\delta)^\ast + mc)(d_{\pm}+n\delta) - 1}{c}\equiv \pm f_\rho\mod\delta \Bigg\}
	\end{align}
	Moreover, if $A:=\begin{pmatrix}
		a & b \\
		c & d
	\end{pmatrix}\in\SL_2(\Z)$ and $\tau\in\HH$ we write 
	$\mu\left(A,\tau\right):=(c\tau+d)^{-1/2}$ and for $A,B\in\SL_2(\Z),$ we define
	$$
	\sigma(A,B):=\frac{\mu(A,B\cdot i)\mu(B,i)}{\mu(AB,i)}.
	$$
	Finally, if $m,n\in\Z,$ define the Kloosterman sum
	$$
	K_c(\rho,\delta,m,n):=\sum\limits_{\begin{pmatrix}
			a & b \\ c & d
		\end{pmatrix}=M\in\mathcal{F}_{L_\rho,\delta}(c)} \frac{\sigma(L_\rho^{-1},M)}{\overline{\varepsilon}_{\delta,g}(L_\rho^{-1}M)}\exp\left(\frac{2\pi i}{c}\left(\frac{(m+\widetilde{\kappa_\rho})a}{t_\rho}+\frac{(n+\widetilde{\kappa_\infty})d}{t_\infty}\right)\right)
	$$
	In this notation, we obtain a Rademacher-type formula for $p_{\delta,g}(n).$ 
	
	\begin{theorem}\label{thm:Rademacher}
		If $\delta\geq 5, 0<g<\delta$ with $g\neq\frac{\delta}{2},$ and $n\geq 1,$ then we have
		\begin{align*}
		p_{\delta,g}(n)&=\frac{3}{2\delta}\pi^{3/2}i^{1/2}\sum\limits_{h_\rho,g_\rho}\overline{\varepsilon_{\delta,g}}(L_\rho^{-1})\alpha_0(-f_\rho g)^{-1}\sum\limits_{0< m\leq \lceil K_\rho\rceil}\left|\frac{\delta(-m+\widetilde{\kappa_\rho})}{t_\rho(n+\widetilde{\kappa_\infty})}\right|^{-3/4}\cdot pp_\rho\left(\frac{\delta}{t_\rho}(-m+\lceil K_\rho\rceil)\right)\\
		&\times\sum\limits_{c>0}\frac{K_c(\rho,\delta,-m,n)}{c}I_{\frac{3}{2}}\left(\frac{4\pi}{c}\sqrt{\frac{|-m+\widetilde{\kappa_\rho}||n+\widetilde{\kappa_\infty}|}{t_\rho\delta}}\right),
		\end{align*}
		where $I_{3/2}$ is the $I$-Bessel function, where $h_\rho$ runs over $\{1,\ldots,\lfloor\frac{\delta}{2}\rfloor\}\cup\{\delta\},$ and where for each $h_\rho,$  $g_\rho$ runs over representatives of $(\Z/\gcd(\delta,h_\rho)\Z)^\times$ with $\gcd(g_\rho,h_\rho)=1,$ and where $e_\rho, f_\rho\in\Z$ are a solution of the equation $e_\rho h_\rho-g_\rho h_\rho=1.$
	\end{theorem}
	
	\begin{remark}
		The description of the cusps of $\Gamma_1(\delta)$ (see Corollary 6.3.19 of \cite{Stromberg}) and the fact that $\Gamma^1(\delta)$ is conjugate to $\Gamma_1(\delta)$ by $S:=\begin{pmatrix}
			0 & -1 \\
			1 & 0
		\end{pmatrix},$ as $h_\rho,g_\rho$ run over the range described in Theorem~\ref{thm:Rademacher}, $\frac{-h_\rho}{g_\rho}$ gives a set of representatives of the cusps of $\Gamma^1(\delta).$ In what follows in this paper, we will write the sums over $\rho$ instead of $h_\rho,g_\rho$ for clarity.
	\end{remark}
	
	To obtain the results in this paper, we make use of the theory of generalized Dedekind eta products and Rankin--Cohen brackets. More precisely, we first show that the generating function of the left-hand side of Theorem~\ref{thm:generalRecurrence} is a modular form of weight $2\nu$ on a congruence subgroup $\Gamma^1(\delta).$ Theorem~\ref{thm:ExRecurDelta5g1} then follows as a special case of writing a modular form in terms of a basis of Eisenstein series and Hecke eigenforms. The more general case is similarly established by using the Petersson inner product, unfolding, and writing the generating function in terms of Poincar\'e series. Writing the generating function as a Poincare series allows us to obtain the Rademacher-type formula by making more explicit the Kloosterman sums in the corresponding Poincar\'e series. The computation of the Petersson inner product yields the infinite sums of values of twisted Dirichlet series in the main recurrence.
	
	The paper is organized as follows. In Section~\ref{sec:GDEP}, we recall the theory of generalized Dedekind eta products and determine their behavior at the cusps of $\Gamma^1(\delta).$ In Section~\ref{sec:PIP}, we recall some facts from harmonic Maass forms and Poincar\'e series and describe our generating functions in terms of Poincar\'e series. We make use of this description to compute the Petersson inner products with cusp forms using the method of unfolding. Finally, in Section~\ref{sec:proofs}, we assemble the results above to prove Theorems~\ref{thm:ExRecurDelta5g1},\ref{thm:generalRecurrence} and \ref{thm:Rademacher}.

	\section*{Acknowledgments}
	
	The first named author was supported by Mamdouha El- Sayed Bobst Deanship Fund, Faculty of Arts and Sciences, AUB.	

	\section{Generalized Dedekind Eta Products}\label{sec:GDEP}
	
	In this section, we recall relevant facts about the generating function $\eta_{\delta,g}$ of $p_{\delta,g}(\cdot)$ and determine the principal parts at the cusps. Moreover, we recall some facts on the congruence subgroup $\Gamma^1(\delta)$ and Eisenstein series.
	
	\subsection{Generalized Dedekind Eta Functions}
	
	One of the key steps in this paper is to note that the generating function of $p_{\delta,g}$ is essentially a congruence modular form of weight $-\frac{1}{2}.$ We first recall these modularity results.
	
	\begin{definition}
		If $\delta\geq 1$ and $g,h\in\Z,$ then we define the generalized Dedekind eta function
		$$
		\eta_{\delta,g,h}(\tau)=\alpha_0(\delta,g,h)e^{\pi i\cdot P_2\left(\frac{g}{\delta}\right)}\tau\prod\limits_{\substack{m > 0 \\ m\equiv \pm g\pmod\delta}}(1-\zeta_\delta^{\pm h}q^{\frac{m}{\delta}}),
		$$
		where $\tau\in\HH$ and $q:=e^{2\pi i\tau}.$ Moreover, if $g\not\equiv 0, \frac{\delta}{2}\pmod\delta,$ then we write 
		$$
		\eta_{\delta,g}^\ast(\tau):=\eta(\tau)\eta_{\delta,g,0}(\tau) = q^{\frac{1}{24}}\prod\limits_{m>0}(1-q^m)\cdot \eta_{\delta,g,0}(\tau).
		$$
	\end{definition}
	
	The function $\eta_{\delta,g}^\ast(\tau)$ transforms as a modular form of weight $\frac{1}{2}$ with a multiplier system. To make this precise, we first recall some notation from the theory of modular forms (see Chapter 4 of \cite{HMFBook}). Let $\leg{c}{d}$ be the extended Legendre symbol and $\sqrt{\cdot}$ be the principal branch of the square root. If $d\in 2\Z+1,$ write
	$$
	\varepsilon_d:=\begin{cases}
		1 & \text{ if }d\equiv 1\pmod 4 \\
		i & \text{ if }d\equiv 3\pmod 4.
	\end{cases}
	$$ Furthermore, if $\mathbb{H}$ is the upper half of the complex plane, $k\in\frac{1}{2}\Z,$ and $f:\HH\to\C$ is a smooth function, then for all $A=\begin{pmatrix}
		a & b \\
		c & d
	\end{pmatrix}\in\GL_2^+(\Z),$ we define
	$$
	(f|_k A)(\tau):=\begin{cases}
		\det(A)^{\frac{k}{2}}(c\tau+d)^{-k}f\left(\frac{a\tau+b}{c\tau+d}\right) & \text{ if }k\in\Z \\
		\leg{c}{d}\varepsilon_d^{2k}(c\tau+d)^{-k}f\left(\frac{a\tau+b}{c\tau+d}\right) & \text{ if }k\in\frac{1}{2}+\Z.
	\end{cases}
	$$
	
	In this notation, we have the following modularity property.
	
	\begin{theorem}\label{thm:EtaTrans}
		If $A=\begin{pmatrix}
			a & b \\
			c & d
		\end{pmatrix}\in\SL_2(\Z),$ then we have
		$$
		(\eta^{\ast}_{\delta,g}|_{1/2}A)(\tau)=\varepsilon_\eta(A)e^{\pi i\mu_{\delta,g}(A)}\eta(\tau)\eta_{\delta,ag,bg}(\tau).
		$$
	\end{theorem}
	
	\begin{proof}
		This follows immediately from (35) of Chapter VIII of \cite{Schoenberg} and that $\eta$ is a modular form of weight $\frac{1}{2}$ with multiplier system $\varepsilon_\eta.$
	\end{proof}
	
	When $A$ is restricted to $\pm\Gamma^1(\delta),$ we obtain modularity.
	
	\begin{corollary}\label{cor:EtaTransSpec}
		If $0<g<\delta$ and $g\neq\frac{\delta}{2},$ then $\eta^\ast_{\delta,g}$ is a holomorphic modular form of weight $\frac{1}{2}$ on the  subgroup 
		$
		\pm\Gamma^1(\delta)
		$
		with multiplier system
		$$
		\varepsilon_{\delta,g}(A):=\varepsilon_\eta(A)\cdot e^{\pi i\mu_{\delta,g}(A)}.
		$$
	\end{corollary}
	
	\begin{remark}
		Note that the subgroup $\Gamma^1(\delta)$ is the conjugate of the more well-known congruence subgroup 
		$$
		\Gamma_1(\delta):=\left\{\pm\begin{pmatrix}
			a & b \\
			c & d
		\end{pmatrix}: a\equiv d\equiv 1\pmod\delta, c\equiv 0\pmod\delta\right\}
		$$
		by the inversion $S:=\begin{pmatrix}
			0 & -1 \\
			1 & 0
		\end{pmatrix}.$
	\end{remark}
	
	The Rankin--Cohen brackets we compute involve $\eta_{\delta,g}^\ast$ and its inverse. In order to compute those Rankin--Cohen brackets in terms of Eisenstein series and cusp forms, we require the principal parts of $\frac{1}{\eta_{\delta,g}}^\ast(\tau).$ To this end, we first compute the cusp widths and parameters with respect to the multiplier system $\overline{\varepsilon_{\delta,g}(\cdot)}.$ In the rest of this paper, if $\rho\in\Q\cup\{\infty\},$ then $L_\rho$ denotes an element of $\SL_2(\Z)$ such that $L_\rho^{-1}\infty=\rho.$ 
	
	\begin{lemma}
		If $\rho:=\frac{-h_\rho}{g_\rho}$ is a cusp of $\pm\Gamma^1(\delta),$ then the cusp width is $t_\rho:=\frac{\delta}{\gcd(\delta,h_\rho)}$ and the cusp parameter is $\kappa_\rho=\{K_\rho\}$ where
		$$
		K_\rho:=\left(\frac{P_2\left(\frac{h_\rho g}{\delta}\right)}{2}+\frac{1}{24}\right)t_\rho.
		$$
	\end{lemma}
	
	\begin{proof}
		A straightforward induction argument shows that if $n\geq 1$ and $L_\rho=\begin{pmatrix}
			e_\rho & f_\rho \\
			g_\rho & h_\rho
		\end{pmatrix}\in\SL_2(\Z),$ then we have that
	$$
	L_\rho^{-1}\begin{pmatrix}
		1 & n \\
		0 & 1
	\end{pmatrix}L_\rho = \begin{pmatrix}
		1+ng_\rho h_\rho  & nh_\rho^2 \\
		-ng_\rho^2 & 1-ng_\rho h_\rho
	\end{pmatrix}.
	$$
	This matrix lies in $\pm\Gamma^1(\delta)$ if and only if $nh_\rho(h_\rho)$ and $nh_\rho(g_\rho)$ are both divisible by $\delta.$ This implies that $nd(e_\rho h_\rho-f_\rho h_\rho)=nh_\rho$ is divisible by $\delta,$ or equivalently, $\frac{\delta}{(\delta,h_\rho)}\mid n.$ The other implication is obvious and this gives our claim for $t_\rho.$ To obtain $\kappa_\rho,$ note that Theorem~\ref{thm:EtaTrans} gives that\footnote{$\doteq$ means equality up to a multiplicative constant} 
	$$
	\left(\eta_{\delta,g}^\ast|_{L_\rho^{-1}}\right)(\tau) \doteq q^{\frac{1}{24}+\frac{P_2\left(\frac{h_\rho g}{\delta}\right)}{2}}\left(1+\sum\limits_{n\geq 1}a(n)q^{\frac{n}{\delta}} \right)
	$$
	with $a(n)\in\C.$ By definition of the cusp width and parameter, this is of the form
	$$
	\left(\eta_{\delta,g}^\ast|_{L_\rho^{-1}}\right)(\tau)=\sum\limits_{m\geq 0}b(m)q^{\frac{m+\kappa_\rho}{t_\rho}}.
	$$
	The claim for $\kappa_\rho$ then follows immediately by comparing the first power of $q$ and taking fractional part on both sides of the exponent.
	\end{proof}
	
	We can now determine the principal part of $\frac{1}{\eta_{\delta,g}^\ast}$ at any cusp $L_\rho^{-1}\infty.$
	
		\begin{lemma}\label{lem:princPartComp}
		If $0<g<\delta$ with $g\neq\frac{\delta}{2}$ and $\frac{-h_\rho}{g_\rho}$ with $g_\rho\neq 0$ and $(g_\rho,h_\rho)=1$ is a cusp, then with $L_\rho$ as above, we have that
		$$
		\varepsilon_{\delta,g}(L_\rho^{-1})\cdot \left(\frac{1}{\eta_{\delta,g}^\ast}\Bigg|_{-1/2}{L_\rho^{-1}}\right)(\tau) = \alpha_0(-f_\rho g)^{-1}\cdot \sum\limits_{0<m\leq \lceil K_\rho\rceil } pp_\rho\left(\frac{\delta}{t_\rho}\left(-m+\lceil K_\rho\rceil\right)\right)q^{\frac{-m+\widetilde{\kappa_\rho}}{t_\rho}} + O(1),
		$$
		where $\widetilde{\kappa_\rho}=\{-\kappa_\rho\}$ and 
		$$
		pp_\rho(l) = \begin{cases}
			\zeta_{\delta}^{-\frac{lf_\rho g}{\alpha_\rho}}\delta_{\alpha_\rho\mid l}+\zeta_{\delta}^{\frac{lf_\rho g}{\beta_\rho}}\delta_{\beta_\rho\mid l} & \text{ if }\alpha\neq 0\text{ and }l\neq 0 \\
			1 & \text{ if }l = 0 \\
			0 & \text{ otherwise} 
		\end{cases},
		$$
		with $\alpha_\rho,\beta_\rho\in\{0,1,\ldots,\delta-1\}$ the representatives of $h_\rho g,-h_\rho g$ in $\Z/\delta\Z.$ 
	\end{lemma}
	
	\begin{proof}
	By Theorem~\ref{thm:EtaTrans}, we have that
	$$
	\varepsilon_{\delta,g}(L_\rho^{-1})\cdot \left(\frac{1}{\eta_{\delta,g}^\ast}\Bigg|_{-1/2}{L_\rho^{-1}}\right)(\tau) = \frac{1}{\eta(\tau)\cdot \eta_{\delta,h_\rho g,-f_\rho g}(\tau)}.
	$$
	Expanding this, we have that
	$$
	\varepsilon_{\delta,g}(L_\rho^{-1})\cdot \left(\frac{1}{\eta_{\delta,g}^\ast}\Bigg|_{-1/2}{L_\rho^{-1}}\right)(\tau) = q^{-1/24}\sum\limits_{n\geq 0}p(n)q^n \cdot\alpha_0(-f_\rho g)^{-1}\cdot q^{-\frac{P_2\left(\frac{h_\rho g}{\delta}\right)}{2}}\cdot\frac{1}{\prod\limits_{m\equiv \pm h_\rho g\mod\delta}(1-\zeta_\delta^{\mp f_\rho g}q^{\frac{m}{\delta}})}.
	$$
	Since $|P_2(x)|\leq \frac{1}{6}$ for all $x\in\RR,$ we clearly have that
	$$
	\varepsilon_{\delta,g}(L_\rho^{-1})\cdot \left(\frac{1}{\eta_{\delta,g}^\ast}\Bigg|_{-1/2}{L_\rho^{-1}}\right)(\tau) = \alpha_0(-f_\rho g)^{-1}\cdot q^{-\frac{P_2\left(\frac{h_\rho g}{\delta}\right)}{2}-\frac{1}{24}}\cdot\frac{1}{\prod\limits_{m\equiv \pm h_\rho g\mod\delta}(1-\zeta_\delta^{\mp f_\rho g}q^{\frac{m}{\delta}})}+O(1).
	$$
	Using the geometric series expansion of the products, we have that
	\begin{align*}
	&	\varepsilon_{\delta,g}(L_\rho^{-1})\cdot \left(\frac{1}{\eta_{\delta,g}^\ast}\Bigg|_{-1/2}{L_\rho^{-1}}\right)(\tau) = \alpha_0(-f_\rho g)^{-1}\cdot q^{-\frac{P_2\left(\frac{h_\rho g}{\delta}\right)}{2}-\frac{1}{24}}\\ 
			&\cdot\prod\limits_{n_1=0}^\infty\sum\limits_{k_1=0}^\infty\zeta_{\delta}^{-k_1f_\rho g}q^{k_1\cdot\left(\frac{\alpha_\rho}{\delta}+n_1\right)}\cdot \prod\limits_{n_2=0}^\infty\sum\limits_{k_2=0}^\infty\zeta_{\delta}^{k_2f_\rho g}q^{k_2\cdot\left(\frac{\beta_\rho}{\delta}+n_2\right)} + O(1).
		\end{align*}
		In order to simplify this, note that if $\alpha_\rho\neq 0,$ then for any choice $k_1\geq 1, k_2\geq 1,$ we have that
		$$
		k_1\left(\frac{\alpha_\rho}{\delta}+n_1\right)+k_2\left(\frac{\beta_\rho}{\delta}+n_2\right)\geq \frac{\alpha_\rho+\beta_\rho}{\delta}\geq 1 > \frac{P_2\left(\frac{h_\rho g}{\delta}\right)}{2}+\frac{1}{24}.
		$$
		The same inequality holds if $k_1n_1\geq 1$ or $k_2n_2\geq 1$ and this gives that
		\begin{equation}\label{eq:EtaGenTransAppFourier}
		\varepsilon_{\delta,g}(L_\rho^{-1})\cdot \left(\frac{1}{\eta_{\delta,g}^\ast}\Bigg|_{-1/2}{L_\rho^{-1}}\right)(\tau) = \alpha_0(-f_\rho g)^{-1}\cdot q^{-\frac{P_2\left(\frac{h_\rho g}{\delta}\right)}{2}-\frac{1}{24}}\cdot\sum\limits_{0\leq m<\frac{\delta}{t_\rho}K_\rho} pp_\rho(m)q^{\frac{m}{\delta}} + O(1).
		\end{equation}
		The claim is then obtained by a change of variables.
	\end{proof}
	
	\subsection{Eisenstein Series for$\Gamma^1(\delta)$}
	
	In order to compute the Rankin--Cohen brackets explicitly, we require a basis of the Eisenstein forms on $\Gamma^1(\delta)$. Here, we give an appropriate construction of this basis.
	
	\begin{lemma}\label{lem:EisensteinBasis}
		If $k\geq 4$ is even, the Eisenstein part of weight $k$ on $\Gamma^1(\delta)$ is spanned by the basis
		$$
		\tilde{G_k}(\chi_1,\chi_2,e)(\tau):=\chi_2(-1)N(\chi_1)^{-k}e^{-k}\left(\delta_{N(\chi_1)=1}L(\chi_2,k)+\left(\frac{2\pi i}{N(\chi_2)}\right)^k\cdot\frac{g(\chi_2)}{(k-1)!}\sum\limits_{n\geq 1}\sigma_{k-1}(\overline{\chi_2},\overline{\chi_1},n)q^{\frac{n}{eN(\chi_1)N(\chi_2)}}\right),
		$$
		where
		$$
		\sigma_{k-1}(\overline{\chi_2},\overline{\chi_1},n)=\sum\limits_{d\mid n}\overline{\chi_2}(d)\overline{\chi_1}(n/d)d^{k-1},
		$$
		where $\chi_1,\chi_2$ range over the primitive characters and $e$ ranges over the integers with the conditions that $\chi_1\chi_2(-1)=1$ and $N(\chi_1)N(\chi_2)e\mid\delta,$ and where
		$$
		g(\chi):=\sum\limits_{t\in(\Z/N(\chi)\Z)^\ast}\chi(t)\zeta_{N(\chi)}^{-t}
		$$
		is a Gauss sum. Moreover, the value of $\tilde{G}_k(\chi_1,\chi_2,e)$ at the cusp $\frac{-h_\rho}{g_\rho}$ is given by
		$$
		\tilde{v}_{\chi_1,\chi_2,e}(\frac{h_\rho}{-g_\rho})=\begin{cases}
			0 & \text{ if }N(\chi_1)\nmid h_\rho \\
			\left(\frac{e}{\gcd(e,h_\rho/N(\chi_1))}\right)^{-k}\overline{\chi_1}\left(\frac{eg_\rho}{\gcd(e,h_\rho/N(\chi_1))}\right)\chi_2\left(\frac{-h_\rho}{N(\chi_1)\cdot\gcd(e,h_\rho/N(\chi_1)) }\right)L(\overline{\chi_1}\chi_2,k) & \text{ if }N(\chi_1)\mid h_\rho.
		\end{cases}
		$$
	\end{lemma}
	
	\begin{proof}
		By Theorem 8.5.17 of \cite{Stromberg}, a basis for the Eisenstein part of weight $k$ on $\Gamma_1(\delta)$ is given by
		$$
		G_k(\chi_1,\chi_2)(e\tau):=\delta_{N(\chi_2)=1}L(\overline{\chi_1},k)+\left(\frac{-2\pi i}{N(\chi_1)}\right)^k\frac{g(\overline{\chi_1})}{(k-1)!}\sum\limits_{n\geq 1}\sigma_{k-1}(\chi_1,\chi_2,n)q^{en}.
		$$		
		Moreover, by Proposition 8.5.3 of \cite{Stromberg}, these Eisenstein series satisfy the functional equation
		$$
		G_k(\chi_1,\chi_2)|W_{\delta_1\delta_2}=\chi_2(-1)\left(\frac{N(\chi_2)}{N(\chi_1)}\right)^{\frac{k}{2}}G_k(\overline{\chi_2},\overline{\chi_1}),
		$$
		where for $N\geq 1, W_N:=\begin{pmatrix}
			0 & -1 \\ 
			N & 0
		\end{pmatrix}$ is the Fricke involution.
		On the other hand, since $\Gamma^1(\delta)=S^{-1}\Gamma_1(\delta)S,$ we have that
		$$
		G_k(\chi_1,\chi_2)|V(e)S = e^{-k} V\left(\frac{1}{e}\right)G_k(\chi_1,\chi_2)|W_{\delta_1\delta_2}\cdot\begin{pmatrix}
			\frac{1}{\delta_1\delta_2} & 0 \\
			0 & 1
		\end{pmatrix},
		$$
		where $V(e):=\begin{pmatrix}
			e & 0 \\
			0 & 1
		\end{pmatrix}$
		form a basis for the Eisenstein series for $\Gamma^1(\delta).$ The right-hand side is precisely our $\tilde{G}_k(\chi_1,\chi_2,e)(\tau).$ Moreover, the values at the cusps follow from the corresponding values for $G_k(\chi_1,\chi_2,e)$ (see Proposition 8.5.6 of \cite{Stromberg}).
	\end{proof}
	
	It turns out that the only Eisenstein components that contribute to our Rankin--Cohen bracket are those where $\chi_1$ is the trivial character. In order to have the form required for Theorem~\ref{thm:generalRecurrence}, we rewrite those series in terms of generalized Bernoulli numbers.
	
	\begin{lemma}\label{lem:EisenSimplify}
		If $\delta\geq 5,k\geq 4, \varepsilon$ denotes the trivial character, and $\chi$ is an even primitive character of conductor $N(\chi)\mid\delta,$ then we have that
		$$
		E_k(\chi,e)(\tau):=\frac{(k-1)!e^kN(\chi)^k}{(2\pi i)^kg(\chi)}\tilde{G}_k(\varepsilon,\chi,e)(\tau)=\frac{-1}{2}\cdot\frac{B_{k,\overline{\chi}}}{k}+\sum\limits_{n\geq 1}\sigma_{k-1}(\overline{\chi},n)q^{\frac{n}{eN(\chi)}}
		$$
		and at the cusp $\frac{-h_\rho}{g_\rho},$ the Eisenstein series $E_k(\chi,e)$ has the value
		$$
		v_{\chi,e}\left(\frac{h_\rho}{-g_\rho}\right) = -\frac{1}{2}\cdot\frac{B_{k,\overline{\chi}}}{k}\cdot  \gcd(e,h_\rho)^k\chi\left(\frac{h_\rho}{\gcd(e,h_\rho)}\right).
		$$
	\end{lemma}
	
	\begin{proof}
		Plugging into the definition of $E_k(\chi,e)$ and $\tilde{G}_k(\varepsilon,\chi,e),$ we clearly have that
		$$
		E_k(\chi,e)(\tau) = \frac{L(\chi,k)(k-1)!N(\chi)^k}{(2\pi i)^kg(\chi)} +  \sum\limits_{n\geq 1}\sigma_{k-1}(\overline{\chi},n)q^{\frac{n}{eN(\chi)}}.
		$$
		On the other hand, since $\chi$ is an even primitive character of conductor $N(\chi),$ we have that (see Theorem 4.15 of \cite{Kowalski} and Proposition 16.6.2 of \cite{Rosen})
		$$
		L(\chi,k)=-\frac{g(\chi)}{2}\cdot\left(\frac{2\pi i}{N(\chi)}\right)^k\cdot\frac{B_{k,\overline{\chi}}}{k!}.
		$$
		Putting this together with Lemma~\ref{lem:EisensteinBasis}, we have our claims.
	\end{proof}
	
	\section{Petersson Inner Products}\label{sec:PIP}
	
	The proof of Theorem~\ref{thm:generalRecurrence} makes use of the essential fact that the generating function of the left-hand side is a modular form of weight $2\nu$ on $\Gamma^1(\delta).$ In order to determine the cuspidal part of this modular form, we require the expression of $\frac{1}{\eta_{\delta,g}^\ast}$ as a Poincar\'e series and the theory of Rankin--Cohen brackets. The Petersson inner product with Hecke eigenforms, using the method of unfolding, gives rise to the values of twisted Dirichlet series that appear. In this section, we compute the details.
	
	\subsection{Harmonic Maass form Poincar\'e series}
	
	Here we recall the construction of negative weight harmonic Maass form Poincar\'e series computed by Bringmann and Ono \cite{BringmannOno} and write $\frac{1}{\eta_{\delta,g}^\ast}$ in terms of those Poincar\'e series. 
	
	Those Poincar\'e series are obtained by averaging certain modified Whittaker functions over the action of the modular group modulo translations. To make this precise,  for $s\in\C$ and $y\in\RR\setminus\{0\},$ let
	$$
	\mathcal{M}_s(y):=|y|^{-\frac{k}{2}}M_{\sign(y)\frac{k}{2},s-\frac{1}{2}}(|y|),
	$$
	where $M_{\nu,\mu}$ is the usual $M$-Whittaker function. If $\tau=x+iy,$ we let
	\begin{equation}\label{PoincareComponent}
		\phi_s(\tau):=\mathcal{M}_s(4\pi y)e(x),
	\end{equation}
	where $e(x):=e^{2\pi i x}.$ Finally, if $\Gamma$ is a finite-index subgroup of $\SL_2(\Z)$ and $\rho:=M_\rho^{-1}\infty$ is a cusp of $\Gamma$, let $\Gamma_\rho$ be the stabilizer of $\rho$ in $\Gamma,$ and we let $\widehat{\Gamma}$ and $\widehat{\Gamma}_\rho$ be the uniformizations of $\Gamma$ and $\Gamma_\rho$ respectively. 
	In this notation, if $\varepsilon$ is a multiplier system on $\Gamma,$  $s\in\C,$ and $\Re(s)>1,$ then we define the Poincar\'e series
	$$
	\mathcal{P}_{\rho}(\tau,m,\Gamma,k,s,\varepsilon):=\frac{1}{\Gamma(2-k)}\sum\limits_{M\in\widehat{\Gamma}_\rho\backslash\widehat{\Gamma}}\varepsilon(M)^{-1}\phi_s\left(\left(\frac{-m+\kappa_\rho}{t_\rho}\right) M_\rho\tau\right)\Bigg|_{k}M,
	$$ 
	where $t$ and $\kappa$ are the cusp width parameter respectively, and  $\Gamma(\cdot)$ is the usual Gamma-function. When $s=\frac{2-k}{2},$ these Poincar\'e series are harmonic Maass forms, and we drop the $s$ from the notation.
	
	\begin{lemma}
		If $k<0,$ then the Poincar\'e series $\mathcal{P}_{\rho}(\tau,m,\Gamma,k,\varepsilon)$ is uniformly and absolutely convergent and defines a harmonic Maass form on $\Gamma$ of weight $k$ and multiplier $\varepsilon.$
	\end{lemma}
	
	In \cite{BringmannOno}, the authors determine the Fourier expansions of these Poincar\'e series at the cusps. To make this precise, recall that if $f(\tau)$ is a harmonic Maass form of weight $k$ and multiplier system, then it has a Fourier expansion of the form
	$$
	f(\tau)=f^+(\tau)+f^{-}(\tau):= \sum\limits_{n\gg-\infty}a_f^{+}(n)q^{\frac{n}{t_\infty}}+\sum\limits_{\substack{n\ll\infty \\ n\neq 0}}a_f^{-}(n)\Gamma\left(1-k,-\frac{4\pi n y}{t_\infty}\right)q^\frac{n}{t_\infty},
	$$
	where $\Gamma(s,x)$ is the incomplete Gamma function and $t_\infty$ is the cusp width of $\infty$. The parts $f^+(\tau)$ and $f^{-}(\tau)$ are called the holomorphic and non-holomorphic parts of $f(\tau)$ respectively. $f(\tau)$ has similar expansions at other cusps. 
	
	In order to write $\frac{1}{\eta_{\delta,g}^\ast}$ in terms of Poincar\'e series, we restrict ourselves to the case where $\Gamma=\Gamma^1(\delta)$ and $k=-\frac{1}{2}.$ Since $\frac{1}{\eta_{\delta,g}^\ast}$ is a weakly holomorphic modular form, the combination of Poincar\'e series we obtain has a vanishing non-holomorphic part, so we only describe the holomorphic part and the principal part. 
	
	To this end, we introduce some notation. 
	If $\rho$ is a cusp of $\Gamma^1(\delta)$ and $L_\rho\in\SL_2(\Z)$ with $L_\rho^{-1}\infty, m,n\in\Z $ and $c>0,$  the Fourier coefficients are given in terms of the Kloosterman sums
	$$
	K_c(\rho,\delta,m,n):=\sum\limits_{\begin{pmatrix}
			a & b \\ c & d
	\end{pmatrix}=M\in\mathcal{F}_{L_\rho,\delta}(c)} \frac{\sigma(L_\rho^{-1},M)}{\overline{\varepsilon}(L_\rho^{-1}M)}\exp\left(\frac{2\pi i}{c}\left(\frac{(m+\widetilde{\kappa_\rho})a}{t_\rho}+\frac{(n+\widetilde{\kappa_\infty})d}{t_\infty}\right)\right)
	$$
	where $\mathcal{F}_{L_\rho,\delta}(c)$ consists of all the matrices $M=\begin{pmatrix}
	a &  b \\
	c & d
	\end{pmatrix}\in \pm L_\rho\Gamma^1(\delta)$ with $0\leq d<ct_\infty$ and $0\leq a<ct_\rho.$ 
	
	In this notation, we have the following lemma.
	
	\begin{lemma}\label{lem:PoincareFourier}
		If $\delta\geq 5, L_\rho\in\SL_2(\Z)$ with $\rho=L_\rho^{-1}\infty,$ and $m\geq1,$ then we have that
		$$
		\mathcal{P}_{L_\rho}\left(\tau,m,\pm\Gamma^1(\delta),\frac{-1}{2},\varepsilon\right)=\sum\limits_{n\gg-\infty}a^{+}(n)q^{\frac{n}{\delta}}+\sum\limits_{\substack{n\ll\infty \\ n\neq 0}}a^{-}(n)\Gamma\left(1-k,-\frac{4\pi n y}{\delta}\right)q^\frac{n}{\delta},
		$$
		where for $n>0,$ the coefficient $a^+(n)$ is given by
		$$
		a^+(n)=-i^{5/2}2\pi\Gamma\left(\frac{5}{2}\right)\left|\frac{\delta(-m+\widetilde{\kappa_\rho})}{t_\rho(n+\widetilde{\kappa_\infty})}\right|^{-\frac{3}{4}}\cdot\frac{1}{\delta}\sum\limits_{c>0}\frac{K_c(\rho,\delta,-m,n)}{c}\cdot I_{3/2}\left(\frac{4\pi}{c}\sqrt{\frac{|-m+\widetilde{\kappa_\rho}||n+\widetilde{\kappa_\infty}|}{t_\rho \delta}}\right)
		$$
		with $I_{s}$ being the $I$-Bessel function. Moreover, the principal part at $\rho$ is given by $q^{\frac{-m+\widetilde{\kappa_\rho}}{t_\rho}}$ and is $0$ otherwise.
	\end{lemma}
	
	\begin{proof}
		See the proof of Theorem 3.2 in \cite{BringmannOno}.
	\end{proof}
	
	In this notation, we obtain the description of $\frac{1}{\eta_{\delta,g}^\ast}$ in terms of Poincar\'e series.
	
	\begin{proposition}\label{prop:EtaPoincare}
		If $\delta\geq 5$ and $0<g<\delta$ with $g\neq\frac{\delta}{2},$then we have
		$$
		\frac{1}{\eta_{\delta,g}^\ast}(\tau)=\sum\limits_{\substack{\rho \\ 0< m \leq \lceil K_\rho\rceil }}\overline{\varepsilon_{\delta,g}}(L_\rho^{-1})\alpha_0(-f_\rho g)^{-1}pp_\rho\left(\frac{\delta}{t_\rho}(-m+\lceil K_\rho\rceil)\right)\mathcal{P}_{L_\rho}\left(\tau,m,\pm\Gamma^1(\delta),-\frac{1}{2},\overline{\varepsilon_{\delta,g}}\right),
		$$
		where $\rho$ varies over the cusps of $\pm\Gamma^1(\delta)$ and $L_\rho^{-1}\infty=\rho.$ 
	\end{proposition}
	
	\begin{proof}
		Lemma~\ref{lem:princPartComp} shows that both sides have the same principal parts. The claim immediately follows from Lemma 3.3 of \cite{Gomez}.
	\end{proof}
	
	In order to obtain Theorem~\ref{thm:Rademacher}, we show that the above definition of  $\mathcal{F}_{L_\rho,c}$ coincides with that in (\ref{eq:FLCDef}). To this end, we require the following straightforward lemma.
	
		\begin{lemma}
		Given $L^{-1}=\begin{pmatrix}
			e & f \\
			g & h
		\end{pmatrix}\in\SL_2(\Z),$ for a fixed $c$ and $a,d\in\Z,$ there is a matrix $S=\begin{pmatrix}
			a & b \\
			c & d
		\end{pmatrix}$ such that 
		$$
		L^{-1}S\in\pm\Gamma^1(\delta)
		$$
		if and only if there exists $\sigma\in\{1,-1\}$ such that the following hold
		\begin{enumerate}
			\item $ad\equiv 1\mod c.$
			\item $e\equiv \sigma d\mod\delta.$
			\item If $b=\frac{ad-1}{c},$ then $b\equiv -\sigma f\mod\delta.$ 
		\end{enumerate}
	\end{lemma}
	
	\begin{proof}
		First, suppose that such a matrix $S$ exists. The first condition is clear since $S\in\SL_2(\Z).$ Now, the condition that $L^{-1}S\in\sigma\Gamma^1(\delta)$ implies that
		$$
		bg+dh\equiv \sigma\mod\delta
		$$
		and
		$$
		be+df\equiv 0\mod\delta.
		$$
		Multiplying the second equation by $g,$ we have that
		$$
		0\equiv beg+dfg\equiv beg+d(eh-1)\equiv e(bg+dh)-d\mod\delta
		$$
		which gives (2) since $bg+dh\equiv\sigma\mod\delta.$ Multiplying the second equation by $h$ instead and arguing similarly, we obtain (3).
		
		Now, suppose (1),(2), and (3) hold. Then we clearly have that $bg+dh\equiv \sigma\mod\delta$ and $be+df\equiv 0\mod\delta$ and therefore $L^{-1}S\in\sigma\Gamma^1(\delta).$
	\end{proof}
	
	We then have the following description for $\mathcal{F}_{L_\rho,\delta}.$

	\begin{lemma}
	Write $L_\rho=\begin{pmatrix}
		e_\rho & f_\rho \\
		g_\rho & h_\rho
	\end{pmatrix}.$ For a fixed $c,$ denote by $d_+$ and $d_{-}$ the representatives in $[0,\delta-1]$ of $h_\rho$ and $-h_\rho$ modulo $\delta$ respectively and for any $d$ with $\gcd(c,d)=1,$ denote by $d^\ast$ the unique solution in $[0,c-1]$ of $dd^\ast\equiv 1\mod c$.
	We have that 
		\begin{align*}
		\mathcal{F}_{L_\rho,\delta}(c):=\Bigg\{&\begin{pmatrix}
		(d_{\pm}+n\delta)^\ast + mc & \frac{	((d_{\pm}+n\delta)^\ast + mc)(d_{\pm}+n\delta) - 1}{c} \\
		c & d_{\pm}+n\delta
	\end{pmatrix} : 0\leq n < \frac{c\delta-d_\pm}{\delta}, \gcd(d_\pm+n\delta,c)=1, \\
	&0\leq m < \frac{ct_\rho-(d_{\pm}+n\delta)^\ast}{c}, \frac{((d_{\pm}+n\delta)^\ast + mc)(d_{\pm}+n\delta) - 1}{c}\equiv \pm f_\rho\mod\delta \Bigg\}
		\end{align*}
	\end{lemma}
	
	\begin{proof}
		This is just a restatement of the above lemma.
	\end{proof}
	
	\subsection{Rankin--Cohen brackets and Petersson Inner Products}
	
	In order to prove Theorem~\ref{thm:generalRecurrence}, we first show that the generating function of the left-hand side is a modular form of weight $2\nu$ on $\Gamma^1(\delta).$ To this end, we write this generating function as a Rankin--Cohen bracket of $\eta_{\delta,g}^\ast$ and $\frac{1}{\eta_{\delta,g}^\ast}.$ 
	
	To this end, recall that if $f$ and $g$ are smooth functions defined on $\HH,$  $k,l\in\R$ and $\nu\geq 0,$ then $\nu$th Rankin--Cohen bracket of $f$ and $g$ is
	\begin{equation}\label{RankinCohenDefinition}
	[f,g]_\nu :=  \sum_{\substack{r,s \geq 0 \\ r + s = \nu}} (-1)^r \frac{\Gamma(k+\nu)\Gamma(l+\nu)}{s!r!\Gamma(k+\nu-s)\Gamma(l+\nu-r)}\cdot D^r(f(\tau)) \cdot D^s(g(\tau)),
\end{equation}
where $D:=\frac{1}{2\pi i}\frac{d}{d\tau}.$
	The next proposition describes the modularity of Rankin--Cohen brackets. 
	\begin{proposition}[Proposition 2.37 of \cite{HMFBook}]\label{RankinCohenModularity}
	Let $f$ and $g$ be (not necessarily holomorphic) modular forms on a subgroup $\Gamma$ of $\SL_2(\Z)$ with multiplier systems $\varepsilon_f$ and $\varepsilon_g$ respectively. Then the following are true.
	
	\noindent
	(1)  We have that $[f,g]_\nu$ is modular of weight $k+l+2\nu$ on $\Gamma$ with multiplier system $\varepsilon_f\varepsilon_g.$
	
	\noindent
	(2)  If $\gamma\in\SL_2(\R),$ then under the usual modular slash operator, we have
	$$
	[f|_k\gamma,g|_l\gamma]_\nu=[f,g]_\nu|_{k+l+2\nu}\gamma.
	$$
\end{proposition}

	Our unfolding argument decomposes the Rankin--Cohen bracket $\left[\eta_{\delta,g}^\ast,\frac{1}{\eta_{\delta,g}^\ast}\right]_\nu$ into simpler Rankin--Cohen brackets. To make this precise, if  $m,t\geq 1,$ and $0\leq\kappa<1,$  we define
	$$
	\lam_{m,t,\kappa}(\tau):= \phi_{\frac{5}{4}}\left(\frac{-m+\kappa}{t}\tau\right).
	$$
	In order to compute the Rankin--Cohen brackets, we require the derivatives of $\lambda_{m,t,\kappa}.$ 
	
		\begin{lemma}\label{lem:lamDer}
		If $r\geq 0,$ then we have
		$$
		D^r(\lam_{m,t,\kappa}(\tau))=\left(\frac{m-\kappa}{t}\right)^r\cdot\left(\frac{-3}{2}\right)^{\overline{r}}e\left(\frac{-m+\kappa}{t}x\right)\left(\frac{4\pi(m-\kappa)y}{t}\right)^{\frac{1}{4}-\frac{r}{2}}\cdot M_{\frac{1}{4}-\frac{r}{2},\frac{3}{4}-\frac{r}{2}}\left(\frac{4\pi(m-\kappa)y}{t}\right).
		$$
	\end{lemma}
	
	\begin{proof}
		This follows exactly as in the proof of Lemma 3.7 of \cite{Gomez}
	\end{proof}
	
		The unfolding is given by the following lemma.
	
	\begin{lemma}\label{lem:LemUnfold}
		If $f$ is a cusp form of weight $2\nu$ on $\pm\Gamma^1(\delta)$ with trivial multiplier, then 
		\begin{align*}
		\langle[\frac{1}{\eta_{\delta,g}^\ast},\eta_{\delta,g}^\ast], f\rangle =  &\frac{1}{\Gamma(\frac{5}{2})}\sum\limits_\rho\sum\limits_{0<m\leq \lceil K_\rho\rceil } pp_\rho\left(\frac{\delta}{t_\rho}(-m+\lceil K_\rho\rceil)\right)\\
		&\times\int_0^\infty\int_0^\delta[\lam_{m,t_\rho,\widetilde{\kappa_\rho}},\eta_{\delta,h_\rho g,-f_\rho g}^\ast]_{\nu}(\overline{f}|_{2\nu}L_{\rho}^{-1})(x+iy)y^{2\nu}\frac{dxdy}{y^2},
		\end{align*}
		where $\rho$ ranges over the cusps of $\Gamma^1(\delta)$ and $L_\rho=:\begin{pmatrix}
			e_\rho & f_\rho \\
			g_\rho& h_\rho
		\end{pmatrix}\in\SL_2(\Z)$ with $L_\rho^{-1}\infty=\rho.$ 
	\end{lemma}
	
	\begin{proof}
		Denote $\Gamma=\pm\Gamma^1(\delta).$ By the same argument as in Lemma 5.8 of \cite{Gomez}, we have that
		$$
		[\frac{1}{\eta_{\delta,g}^\ast},\eta_{\delta,g}^\ast]_\nu = \frac{1}{\Gamma(5/2)}\sum\limits_\rho\sum\limits_m pp_\rho\left(\frac{\delta}{t_\rho}(-m+\lceil K_\rho\rceil)\right)\overline{\varepsilon_{\delta,g}}(L_\rho^{-1})\sum\limits_{M\in\Gamma_\rho\backslash\Gamma}[\lam_{m,t_\rho,\widetilde{\kappa_\rho}}|_{-\frac{1}{2}}L_\rho,\eta_{\delta,g}^\ast]|_{2\nu}M.
		$$
		By the definition of the Petersson inner product, we then have that
		$$
		\langle[\frac{1}{\eta_{\delta,g}^\ast},\eta_{\delta,g}^\ast]_\nu,f\rangle = \frac{1}{\Gamma(\frac{5}{2})}\sum\limits_\rho\sum\limits_m pp_\rho\left(\frac{\delta}{t_\rho}(-m+\lceil K_\rho\rceil)\right)\overline{\varepsilon_{\delta,g}}(L_\rho^{-1})\iint_{\mathbb{H}/\Gamma_\rho}[\lam_{m,t_\rho,\widetilde{\kappa_\rho}}|_{-\frac{1}{2}}L_\rho,\eta_{\delta,g}^\ast]_{2\nu}\overline{f}(x+iy)y^{2\nu}\frac{dxdy}{y^2}.
		$$
		Using Theorem~\ref{thm:EtaTrans}, we have that
		$$
		\langle[\frac{1}{\eta_{\delta,g}^\ast},\eta_{\delta,g}^\ast]_\nu,f\rangle =\frac{1}{\Gamma(\frac{5}{2})}\sum\limits_\rho\sum\limits_m pp_\rho\left(\frac{\delta}{t_\rho}(-m+\lceil K_\rho\rceil)\right)\iint_{\mathbb{H}/\Gamma_\rho}[\lam_{m,t_\rho,\widetilde{\kappa_\rho}},\eta_{\delta,h_\rho g, -f_\rho g}^\ast]_{\nu}|_{2\nu}L_\rho\overline{f}|_{2\nu} L_\rho^{-1}|_{2\nu}L_\rho y^{2\nu}\frac{dxdy}{y^2}.
		$$
		Replacing $\tau$ by $L_\rho^{-1}\tau,$ 
		$$
		\langle[\frac{1}{\eta_{\delta,g}^\ast},\eta_{\delta,g}^\ast]_\nu,f\rangle = \frac{1}{\Gamma(\frac{5}{2})}\sum\limits_\rho\sum\limits_m pp_\rho\left(\frac{\delta}{t_\rho}(-m+\lceil K_\rho\rceil)\right)\int_0^\infty\int_0^\delta[\lam_{m,t_\rho,\widetilde{\kappa_\rho}},\eta_{\delta,h_\rho g, -f_\rho g}^\ast]_{\nu}(\overline{f}|_{2\nu}L_{\rho}^{-1})(x+iy)y^{2\nu}\frac{dxdy}{y^2}.
		$$
	\end{proof}
	
		In order to compute these integrals, we require a slightly more general form of (\ref{eq:gEtaExp}).
	
	\begin{lemma}\label{lem:EulerPentagGen}
		If $\delta>2, 0<g<\delta$ with $g\neq\frac{\delta}{2}$ and $0\leq h<\delta,$ then we have
		$$
		\eta_{\delta,g,h}^\ast(\tau):=\eta_{\delta,g,h}(\tau)\cdot\eta(\tau)=\alpha_0(h)\cdot q^{\frac{1}{24}+\frac{P_2(\frac{g}{\delta})}{2}}\cdot\sum\limits_{n=-\infty}^{\infty} (-1)^n\zeta_\delta^{nh}q^{\frac{n^2-n}{2}+\frac{ng}{\delta}}.
		$$
	\end{lemma}
	
	\begin{proof}
		The Jacobi triple product identity is given by
		$$
		\prod\limits_{m=1}^{\infty} (1-x^{2m})(1+x^{2m-1}y^2)(1+\frac{x^{2m-1}}{y^2}) = \sum\limits_{n=-\infty}^\infty x^{n^2}y^{2n}.
		$$
		
		Now, let $x=q^{\frac{1}{2}}$ and $y^2 = -q^{-\frac{1}{2}}\cdot q^{\frac{g}{\delta}}\cdot\zeta_{\delta}^h.$
		
		Plugging this, we obtain
		\begin{align*}
			&\prod\limits_{m=1}^\infty (1-q^\frac{m\delta}{\delta}) \cdot\prod\limits_{m=1}^\infty (1-q^{m-\frac{1}{2}}q^{-\frac{1}{2}}q^{\frac{g}{\delta}}\zeta_\delta^h)\prod\limits_{m=1}^\infty (1-q^{m-\frac{1}{2}}\cdot q^{\frac{1}{2}} q^{-\frac{g}{\delta}}\zeta_{\delta}^{-h}) \\
			&=q^{-\frac{1}{24}}\eta(\tau)\cdot\prod\limits_{m=0}^\infty (1-\zeta_\delta^h\cdot q^{\frac{m\delta+ g}{\delta}})\cdot\prod\limits_{m=1}^\infty (1 - \zeta_{\delta}^{-h}\cdot q^{\frac{m\delta -g}{\delta}}) \\
			&= q^{-\frac{1}{24} - \frac{P_2(\frac{g}{\delta})}{2}}\cdot\alpha_0(h)^{-1}\cdot\eta_{\delta,g,h}^\ast(\tau) \\
			&=\sum\limits_{n=-\infty}^\infty (-1)^n \zeta_{\delta}^{nh}\cdot q^{\frac{n^2-n}{2}+\frac{ng}{\delta}}.
		\end{align*}
		
	\end{proof}
	
		Finally, in order to unfold, we will require the following lemma on the cusp parameters.

	\begin{lemma}\label{lem:cuspParamLemma}
		If $m\ge 1$, $n\in \mathbb Z$, $pp_\rho(0)\neq 0,$  and $\rho=\frac{-h_\rho}{g_\rho}$ is a cusp, then
		\[
		\delta\left(\frac{-m+\widetilde{\kappa_\rho}}{t_\rho}+
		\frac{(\delta-2h_\rho g-2\delta n)^2}{8\delta^2}\right)\in \mathbb Z.
		\]
	\end{lemma}
	
	\begin{proof}
		Let $\alpha_\rho\in\{0,1,\dots,\delta-1\}$ be the representative of $h_\rho g \pmod{\delta}$.
		By (\ref{eq:EtaGenTransAppFourier}), the expansion of $\frac{1}{\eta_{\delta,g}^\ast}$ at the cusp $\rho$ is
		\[
		\left.\frac{1}{\eta^*_{\delta,g}}\right|_{L_\rho^{-1}}(\tau)
		\doteq
		q^{-\frac1{24}-\frac12P_2(\frac{\alpha_\rho}{\delta})}
		\sum_{u} pp_\rho(u) q^{\frac{u}{\delta}} + O(1).
		\]
		Since the cusp width is $t_\rho$ and the cusp parameter with respect to $\overline{\varepsilon_{\delta,g}}$ is $\widetilde{\kappa_\rho}$, all exponents occurring in this expansion lie in
		\[
		\frac{\widetilde{\kappa_\rho}}{t_\rho}+\frac{1}{t_\rho}\mathbb Z.
		\]
		In particular,
		\[
		\frac{\widetilde{\kappa_\rho}}{t_\rho}\equiv
		-\frac1{24}-\frac12P_2\!\left(\frac{\alpha_\rho}{\delta}\right)
		\pmod{\frac1{t_\rho}}.
		\]
		Multiplying by $\delta$ and using $t_\rho\mid \delta$, we obtain
		\[
		\delta\frac{\widetilde{\kappa_\rho}}{t_\rho}\equiv
		-\frac{\delta}{24}-\frac{\delta}{2}P_2\!\left(\frac{\alpha_\rho}{\delta}\right)
		\pmod{\mathbb Z}.
		\]
		Now
		\[
		P_2\!\left(\frac{\alpha_\rho}{\delta}\right)
		=\frac{\alpha_\rho^2}{\delta^2}-\frac{\alpha_\rho}{\delta}+\frac16,
		\]
		so
		\[
		-\frac{\delta}{24}-\frac{\delta}{2}P_2\!\left(\frac{\alpha_\rho}{\delta}\right)
		=
		-\frac{\alpha_\rho^2}{2\delta}+\frac{\alpha_\rho}{2}-\frac{\delta}{8}.
		\]
		Hence
		\[
		\delta\frac{\widetilde{\kappa_\rho}}{t_\rho}
		+\frac{\alpha_\rho^2}{2\delta}-\frac{\alpha_\rho}{2}+\frac{\delta}{8}\in \mathbb Z.
		\]
		Since $\alpha_\rho\equiv h_\rho g \pmod{\delta}$, one checks directly that
		\[
		\frac{\alpha_\rho^2}{2\delta}-\frac{\alpha_\rho}{2}
		\equiv
		\frac{h_\rho^2g^2}{2\delta}-\frac{h_\rho g}{2}
		\pmod{\mathbb Z},
		\]
		and therefore
		\[
		\delta\frac{\widetilde{\kappa_\rho}}{t_\rho}
		+\frac{h_\rho^2g^2}{2\delta}-\frac{h_\rho g}{2}+\frac{\delta}{8}\in \mathbb Z.
		\]
		Finally,
		\[
		\delta\cdot \frac{(\delta-2h_\rho g-2\delta n)^2}{8\delta^2}
		=
		\frac{\delta}{8}-\frac{h_\rho g}{2}+\frac{h_\rho^2g^2}{2\delta}
		+\bigg(\frac{\delta(n^2-n)}{2}+h_\rho g\,n\bigg),
		\]
		and the term in parentheses is an integer. Thus
		\[
		\delta\left(\frac{-m+\widetilde{\kappa_\rho}}{t_\rho}+
		\frac{(\delta-2h_\rho g-2\delta n)^2}{8\delta^2}\right)\in \mathbb Z,
		\]
		as claimed.
	\end{proof}
	
	We now write the unfolding in Lemma~\ref{lem:LemUnfold} as integrals which we will then compute.
	
		\begin{lemma}\label{lem:InnerProdComputation}
		If $\delta\geq 5, 0<g<\delta, g\neq\frac{\delta}{2}$ and $f$ is a cusp form of weight $2\nu$ on $\pm\Gamma^1(\delta),$ with Fourier expansions at $\rho$ given by
		$$
		(f|_{2\nu}L_\rho^{-1})(\tau)=\sum\limits_{n\geq 1} a_{f,\rho}(n)q^{\frac{n}{t_\rho}},
		$$
		 then we have
		\begin{align*}
			&\langle[\frac{1}{\eta_{\delta,g}^\ast},\eta_{\delta,g}^\ast]_\nu,f\rangle = \frac{1}{\Gamma\left(\frac{5}{2}\right)}\sum\limits_{\substack{\rho, m \\ r+s=\nu \\ r,s\geq 0 }}\sum\limits_{n=-\infty}^{\infty}pp_\rho\left(\frac{\delta}{t_\rho}(-m+\lceil K_\rho\rceil)\right)\frac{(-1)^r\Gamma(\nu-\frac{1}{2})\Gamma(\nu+\frac{1}{2})}{s!r!\Gamma(r-\frac{1}{2})\Gamma( s+\frac{1}{2})} \left(\frac{m-\widetilde{\kappa_\rho}}{t_\rho}\right)^r\cdot\left(\frac{-3}{2}\right)^{\overline{r}} \\
			&\times \left(\frac{4\pi(m-\widetilde{\kappa_\rho})}{t_\rho}\right)^{\frac{1}{4}-\frac{r}{2}}\frac{\alpha_0(-f_\rho g)}{(8\delta^2)^s}(-1)^n\zeta_\delta^{-nf_\rho g}(\delta-2h_\rho g-2\delta n)^{2s} \delta \overline{a_{f,\rho}\left(\delta\left(\frac{-m+\widetilde{\kappa_\rho}}{t_\rho}+\frac{(\delta-2h_\rho g-2\delta n)^2}{8\delta^2}\right)\right)} \\
			&\times\int_0^\infty  y^{\frac{1}{4}-\frac{r}{2}}M_{\frac{1}{4}-\frac{r}{2},\frac{3}{4}-\frac{r}{2}}\left(\frac{4\pi(m-\widetilde{\kappa_\rho})y}{t_\rho}\right)\exp\left(-2\pi y\left(\frac{-m+\widetilde{\kappa_\rho}}{t_\rho}+\frac{(\delta-2h_\rho g-2\delta n)^2}{4\delta^2}\right)\right)y^{2\nu-2}dy
		\end{align*}
		
	\end{lemma}
	
	\begin{proof}
		By Lemmas~\ref{lem:LemUnfold} and the definition of the Rankin--Cohen bracket, we have that
		\begin{align*}
		\langle[\frac{1}{\eta_{\delta,g}^\ast},\eta_{\delta,g}^\ast]_\nu,f\rangle = \frac{1}{\Gamma(\frac{5}{2})}&
		\sum\limits_{\rho,m}pp_\rho\left(\frac{\delta}{t_\rho}(-m+\lceil K_\rho\rceil)\right)\\
		&\int_0^\infty\int_0^\delta \sum\limits_{\substack{r+s =\nu \\ r,s\geq 0}}\frac{(-1)^r\Gamma(\nu-\frac{1}{2})\Gamma(\nu+\frac{1}{2})}{s!r!\Gamma(r-\frac{1}{2})\Gamma({ s}+\frac{1}{2})} D^r(\lam_{m,t_\rho,\widetilde{\kappa_\rho}})D^s\eta_{\delta,h_\rho g,-f_\rho g}^\ast \overline{f}|L_\rho^{-1}\cdot y^{2\nu}\frac{dxdy}{y^2}.
		\end{align*}
		Applying Lemmas~\ref{lem:lamDer} and \ref{lem:EulerPentagGen}, we have
		\begin{align*}
			\langle[\frac{1}{\eta_{\delta,g}^\ast},\eta_{\delta,g}^\ast]_\nu,f\rangle = &\frac{1}{\Gamma(\frac{5}{2})}\sum\limits_{\rho,m}pp_\rho\left(\frac{\delta}{t_\rho}(-m+\lceil K_\rho\rceil)\right) \int_0^\infty\int_0^\delta\sum\limits_{\substack{r+s =\nu \\ r,s\geq 0}}\frac{(-1)^r\Gamma(\nu-\frac{1}{2})\Gamma(\nu+\frac{1}{2})}{s!r!\Gamma(r-\frac{1}{2})\Gamma({ s}+\frac{1}{2})} \left(\frac{m-\widetilde{\kappa_\rho}}{t_\rho}\right)^r\cdot\left(\frac{-3}{2}\right)^{\overline{r}} \\
			&\times e\left(\frac{-m+\widetilde{\kappa_\rho}}{t_\rho}x\right)\left(\frac{4\pi(m-\widetilde{\kappa_\rho})y}{t_\rho}\right)^{\frac{1}{4}-\frac{r}{2}}\cdot M_{\frac{1}{4}-\frac{r}{2},\frac{3}{4}-\frac{r}{2}}\left(\frac{4\pi(m-\widetilde{\kappa_\rho})y}{t_\rho}\right) \\
			&\times\frac{\alpha_0(-f_\rho g)}{(8\delta^2)^s}\sum\limits_{n=-\infty}^{\infty}(-1)^n\zeta_{\delta}^{-nf_\rho g}(\delta-2h_\rho g-2\delta n)^{2s}\cdot q^{\frac{(\delta-2h_\rho g-2\delta n)^2}{8\delta^2}}\overline{f}|L_\rho^{-1} y^{2\nu}\frac{dxdy}{y^2}.
		\end{align*}
		We expand $f$ as a Fourier series and then we obtain
		\begin{align*}
			\langle[\frac{1}{\eta_{\delta,g}^\ast},\eta_{\delta,g}^\ast]_\nu,f\rangle =
			& \frac{1}{\Gamma\left(\frac{5}{2}\right)}\sum\limits_{\substack{\rho,m \\ r+s=\nu \\ r,s\geq 0}}pp_\rho\left(\frac{\delta}{t_\rho}(-m+\lceil K_\rho\rceil)\right) \frac{(-1)^r\Gamma(\nu-\frac{1}{2})\Gamma(\nu+\frac{1}{2})}{s!r!\Gamma(r-\frac{1}{2})\Gamma({ s}+\frac{1}{2})} \left(\frac{m-\widetilde{\kappa_\rho}}{t_\rho}\right)^r\cdot\left(\frac{-3}{2}\right)^{\overline{r}} \\
			&\times \frac{\alpha_0(-f_\rho g)}{(8\delta^2)^s}\int_0^\infty\left(\frac{4\pi(m-\widetilde{\kappa_\rho})y}{t_\rho}\right)^{\frac{1}{4}-\frac{r}{2}} M_{\frac{1}{4}-\frac{r}{2},\frac{3}{4}-\frac{r}{2}}\left(\frac{4\pi(m-\widetilde{\kappa_\rho})y}{t_\rho}\right) \\
			&\times \sum\limits_{n=-\infty}^\infty\sum\limits_{j = 1}^\infty (-1)^n\zeta_{\delta}^{-nf_\rho g}(\delta-2h_\rho g-2\delta n)^{2s}\overline{a_{f,\rho}(j)} \\
			&\times\exp\left(-2\pi y\left(\frac{j}{\delta}+\frac{(\delta-2h_\rho g-2\delta n)^2}{8\delta^2}\right)\right)y^{2\nu-2}dy \\
			&\times\int_0^\delta e\left(\left(\frac{-m+\widetilde{\kappa_\rho}}{t_\rho}+\frac{(\delta-2h_\rho g-2\delta n)^2}{8\delta^2} -\frac{j}{\delta}\right)x\right)dx
		\end{align*}
		By Lemma~\ref{lem:cuspParamLemma}, when $pp_\rho(0)\neq0,$  the last integral is $0$ unless
		$$
		j=\delta\left(\frac{-m+\widetilde{\kappa_\rho}}{t_\rho}+\frac{(\delta-2h_\rho g-2\delta n)^2}{8\delta^2}\right)
		$$
		in which case it is $\delta.$ Therefore, our Petersson inner product reduces to
		\begin{align*}
			\langle[\frac{1}{\eta_{\delta,g}^\ast},\eta_{\delta,g}^\ast]_\nu,f\rangle =
			& \frac{1}{\Gamma\left(\frac{5}{2}\right)}\sum\limits_{\substack{\rho,m \\ r+s=\nu \\ r,s\geq 0}}pp_\rho\left(\frac{\delta}{t_\rho}(-m+\lceil K_\rho\rceil)\right) \frac{(-1)^r\Gamma(\nu-\frac{1}{2})\Gamma(\nu+\frac{1}{2})}{s!r!\Gamma(r-\frac{1}{2})\Gamma({ s}+\frac{1}{2})} \left(\frac{m-\widetilde{\kappa_\rho}}{t_\rho}\right)^r\cdot\left(\frac{-3}{2}\right)^{\overline{r}} \\
			&\times \left(\frac{4\pi(m-\widetilde{\kappa_\rho})}{t_\rho}\right)^{\frac{1}{4}-\frac{r}{2}}\frac{\alpha_0(-f_\rho g)}{(8\delta^2)^s}\int_0^\infty y^{\frac{1}{4}-\frac{r}{2}} M_{\frac{1}{4}-\frac{r}{2},\frac{3}{4}-\frac{r}{2}}\left(\frac{4\pi(m-\widetilde{\kappa_\rho})y}{t_\rho}\right) \\
			&\times\sum\limits_{n=-\infty}^\infty (-1)^n\zeta_\delta^{-nf_\rho g}(\delta-2h_\rho g-2\delta n)^{2s} \overline{a_{f,\rho}\left(\delta\left(\frac{-m+\widetilde{\kappa_\rho}}{t_\rho}+\frac{(\delta-2h_\rho g-2\delta n)^2}{8\delta^2}\right)\right)} \\
			&\times\delta\exp\left(-2\pi y\left(\frac{-m+\widetilde{\kappa_\rho}}{t_\rho}+\frac{(\delta-2h_\rho g-2\delta n)^2}{4\delta^2}\right)\right)y^{2\nu-2}dy
		\end{align*}
		which is our claim.
	\end{proof}

		In order to compute this infinite sum, we require a few lemmas. First, we compute the integral in the above lemma.
	
	\begin{lemma}\label{lem:IntComputation}
		If $r,s\geq 0, \nu=r+s$ and $n,m\geq 0,$ then we have 
		\begin{align*}
			&\left(\frac{m-\widetilde{\kappa_\rho}}{t_\rho}\right)^r\cdot\left(\frac{(\delta-2h_\rho g-2\delta n)^2}{8\delta^2}\right)^s\\
			& \times \int_0^\infty  \left(\frac{4\pi(m-\widetilde{\kappa_\rho})y}{t_\rho}\right)^{\frac{1}{4}-\frac{r}{2}}M_{\frac{1}{4}-\frac{r}{2},\frac{3}{4}-\frac{r}{2}}\left(\frac{4\pi(m-\widetilde{\kappa_\rho})y}{t_\rho}\right)\exp\left(-2\pi y\left(\frac{-m+\widetilde{\kappa_\rho}}{t_\rho}+\frac{(\delta-2h_\rho g-2\delta n)^2}{4\delta^2}\right)\right)y^{2\nu-2}dy \\
			&=\frac{2^{\frac{7}{2}}\delta^{2\nu+1}}{\pi^{2\nu-1}}\left(\frac{m-\widetilde{\kappa_\rho}}{t_\rho}\right)^{\frac{3}{2}}\cdot\frac{\Gamma\left(2\nu-r+\frac{1}{2}\right)}{|\delta-2h_\rho g-2\delta n|^{2\nu+1}}{_2F_1}\left(\begin{array}{cc}
				1 & \frac{1}{2}+2\nu-r\\
				~& \frac{5}{2}-r \\
			\end{array}\mid \frac{8\delta^2(m-\widetilde{\kappa_\rho})}{t_\rho(\delta-2h_\rho g-2\delta n)^2}\right),
		\end{align*}
		where ${_2F_1}$ is the classical hypergeometric function.
	\end{lemma}
	
	\begin{proof}
		Applying the change-of-variables, $y=\frac{t_\rho}{4\pi(m-\widetilde{\kappa_\rho})}t,$ the left-hand side is equal to
		$$
		\frac{1}{(4\pi)^{2\nu-1}}\cdot\left(\frac{(\delta-2h_\rho g-2\delta n)^2}{8\delta^2}\right)^s\cdot \left(\frac{t_\rho}{m-\widetilde{\kappa_\rho}}\right)^{2\nu-1-r}\int_0^\infty t^{2\nu-\frac{r}{2}-\frac{7}{4}}M_{\frac{1}{4}-\frac{r}{2},\frac{3}{4}-\frac{r}{2}}(t)\exp\left(-\left(z-\frac{1}{2}\right)t\right)dt,
		$$
		where
		$$
		z=\frac{t_\rho(\delta-2h_\rho g-2\delta n)^2}{8\delta^2(m-\widetilde{\kappa_\rho})}.
		$$
		Using the identity in Lemma 3.11 of \cite{Gomez} with $\lambda=2\nu-\frac{r}{2}-\frac{3}{4}, \kappa=\frac{1}{4}-\frac{r}{2},\mu=\frac{3}{4}-\frac{r}{2},$ the integral is equal to
		$$
		\frac{1}{(4\pi)^{2\nu-1}}\cdot\left(\frac{(\delta-2h_\rho g-2\delta n)^2}{8\delta^2}\right)^s\cdot \left(\frac{t_\rho}{m-\widetilde{\kappa_\rho}}\right)^{2\nu-1-r}\cdot\frac{\Gamma\left(2\nu-r+\frac{1}{2}\right)}{z^{2\nu-r+\frac{1}{2}}}{_2F_1}\left(\begin{array}{cc}
			1 & \frac{1}{2}+2\nu-r\\
			~& \frac{5}{2}-r \\
		\end{array}\mid \frac{1}{z}\right).
		$$
		Plugging the expression of $z$ and using $\nu=r+s,$ this gives our expression.
	\end{proof}
	
	Putting Lemmas~\ref{lem:InnerProdComputation} and \ref{lem:IntComputation} together, we obtain that
	\begin{align*}
		\langle[\frac{1}{\eta_{\delta,g}^\ast},\eta_{\delta,g}^\ast]_\nu,f\rangle =
		&\frac{2^{\frac{7}{2}}\delta^{2\nu+2}}{\Gamma\left(\frac{5}{2}\right)\pi^{2\nu-1}}
		\sum\limits_{\substack{\rho, m \\ r+s=\nu \\ r,s\geq 0}}\left(\frac{m-\widetilde{\kappa_\rho}}{t_\rho}\right)^{\frac{3}{2}}pp_\rho\left(\frac{\delta}{t_\rho}(-m+\lceil K_\rho\rceil)\right)\alpha_0(-f_\rho g) \\
		&\sum\limits_{n=-\infty}^\infty (-1)^n\zeta_\delta^{-nf_\rho g} \overline{a_{f,\rho}\left(\delta\left(\frac{-m+\widetilde{\kappa_\rho}}{t_\rho}+\frac{(\delta-2h_\rho g-2\delta n)^2}{8\delta^2}\right)\right)}\cdot\frac{1}{|\delta-2h_\rho g-2\delta n|^{2\nu+1}}\omega(\rho, m,\nu,n),
	\end{align*}
	where
	$$
	\omega(\rho,m,\nu,n)=\sum\limits_{r=0}^{\nu}\frac{(-1)^r\Gamma(\nu-\frac{1}{2})\Gamma(\nu+\frac{1}{2})\Gamma\left(2\nu-r+\frac{1}{2}\right)}{(\nu-r)!r!\Gamma(r-\frac{1}{2})\Gamma(\nu-r+\frac{1}{2})}\left(\frac{-3}{2}\right)^{\overline{r}}{_2F_1}\left(\begin{array}{cc}
		1 & \frac{1}{2}+2\nu-r\\
		~& \frac{5}{2}-r \\
	\end{array}\mid \frac{8\delta^2(m-\widetilde{\kappa_\rho})}{t_\rho(\delta-2h_\rho g-2\delta n)^2}\right).
	$$
	
		We will require the following simplification of $\omega(\rho,m,\nu,n).$ 
	
	\begin{lemma}\label{lem:omegaSimplify}
		We have that
		$$
		\omega(\rho,m,\nu,n)=(1-z)^{1-2\nu}\Gamma\left(\nu-\frac{1}{2}\right)\Gamma\left(\nu+\frac{1}{2}\right)\sum\limits_{i=0}^{\nu-2}(-1)^iz^i\binom{2\nu-2}{i}\frac{(\nu-i-1)^{\overline{\nu}}\left(\frac{3}{2}\right)^{\overline{i}}}{\left(-\frac{1}{2}-i\right)^{\overline{\nu}}\Gamma\left(-\frac{1}{2}\right)\left(\frac{5}{2}\right)^{\overline{i}}},
		$$
		where
		$$
		z=\frac{8\delta^2(m-\widetilde{\kappa_\rho})}{t_\rho(\delta-2h_\rho g-2\delta n)^2}.
		$$
	\end{lemma}
	
	\begin{proof}
		This is Lemmas 3.12 and 3.13 of \cite{Gomez}.
	\end{proof}
	
	Putting all this together, we compute the Petersson inner product of $\left[\frac{1}{\eta_{\delta,g}^\ast},\eta_{\delta,g}^\ast\right]_\nu$ with cusp forms on $\Gamma^1(\delta).$ 
	
	\begin{proposition}\label{prop:PeterssonInner}
		If $\delta\geq 5, 0<g<\delta$ with $g\neq\frac{\delta}{2},\nu\geq 2,$ and $f$ is a cusp form of weight $2\nu$ on $\Gamma^1(\delta)$ with Fourier coefficients at cusp $\rho=L_\rho^{-1}\infty=\begin{pmatrix}
			e_\rho & f_\rho \\
			g_\rho & h_\rho
		\end{pmatrix}^{-1}\infty$ written as
		$$
		(f|_{2\nu}L_\rho^{-1})(\tau) = \sum\limits_{n\geq 1}a_{f,\rho}(n)q^{\frac{n}{t_\rho}},
		$$
		then we have that 
		$$
		\langle[\frac{1}{\eta_{\delta,g}^\ast},\eta_{\delta,g}^\ast]_\nu,f\rangle = \sum\limits_{\rho,m}\left(\frac{8}{\beta_{\rho,m}}\right)^{\frac{3}{2}}pp_\rho\left(\frac{\delta}{t_\rho}(-m+\lceil K_\rho\rceil)\right)\alpha_0(-f_\rho g)D_{f,\delta,g,\rho,m}.
		$$
	\end{proposition}

	Finally, in order to obtain the expression in Theorem~\ref{thm:generalRecurrence}, we compute the Fourier expansion of $[\frac{1}{\eta_{\delta,g}^\ast},\eta_{\delta,g}^\ast]_\nu.$
	
	\begin{lemma}\label{lem:FourierBracket}
		If $\nu\geq 1,\delta\geq 5$ and $0<g<\delta,$ then we have that 
		$$
		[\frac{1}{\eta_{\delta,g}^\ast},\eta_{\delta,g}^\ast]_\nu=\sum\limits_{\substack{l\geq 0 \\ k\in\Z}} g_\nu(\delta,g;l,k)(-1)^kp(\delta,g;l-\omega(\delta,g;k))q^{\frac{l}{\delta}},
		$$
		where $g_\nu(\delta,g,l,k)$ is as in Theorem~\ref{thm:generalRecurrence}.
	\end{lemma}
	
	\begin{proof}
		This follows immediately from the definition of the Rankin--Cohen bracket and Lemma~\ref{lem:EulerPentagGen} with $h=0.$ 
	\end{proof}

	\section{Proofs of Theorems~\ref{thm:ExRecurDelta5g1}, \ref{thm:generalRecurrence}, and \ref{thm:Rademacher}}\label{sec:proofs}
	
	\begin{proof}[Proof of Theorem~\ref{thm:ExRecurDelta5g1}]
		We have that $[\frac{1}{\eta_{5,g}^\ast},\eta_{5,g}^\ast]_\nu$ is a modular form of weight $2\nu$ on $\Gamma^1(5).$ If $\nu=2,$ $M_4(\Gamma^1(5))$ is spanned by Eisenstein series and a cusp form $f.$ Moreover, $S_4(\Gamma_1(5))$ is spanned by $g=\sum\limits_{n=1}^\infty b(n)q^n$ and $g$ satisfies
		$$
		g|W_\delta = g.
		$$
		However, we have that
		$$
		g|\begin{pmatrix}
			0 & -1 \\
			1 & 0
		\end{pmatrix}=g|W_\delta\begin{pmatrix}
		\frac{1}{\delta} & 0 \\
		0 & 1
		\end{pmatrix}.
		$$
		Therefore, we can choose
		$$
		f(\tau)=\sum\limits_{n=1}^{\infty} b(n)q^{\frac{n}{\delta}}.
		$$
		Our claim follows by writing out the first few Fourier coefficients.
	\end{proof}
	
	\begin{proof}[Proof of Theorem~\ref{thm:generalRecurrence}]
		We first obtain the constant terms of $[\frac{1}{\eta_{\delta,g}^\ast},\eta_{\delta,g}^\ast]_\nu$ at each cusp. By Theorem~\ref{thm:EtaTrans}, if $L_\rho=\begin{pmatrix}
			e_\rho & f_\rho \\
			g_\rho & h_\rho
		\end{pmatrix},$ this is the constant component of
		\begin{align*}
			&\Bigg[\alpha_0(-f_\rho g)^{-1}\cdot q^{-\frac{P_2\left(\frac{h_\rho g}{\delta}\right)}{2}-\frac{1}{24}}\cdot\sum\limits_{n\geq 0}p(n)q^n\prod\limits_{n_1=0}^\infty\sum\limits_{k_1=0}^\infty\zeta_{\delta}^{-k_1bg}q^{k_1\cdot\left(\frac{\alpha}{\delta}+n_1\right)}\cdot \prod\limits_{n_2=0}^\infty\sum\limits_{k_2=0}^\infty\zeta_{\delta}^{k_2f_\rho g}q^{k_2\cdot\left(\frac{\beta}{\delta}+n_2\right)},\\ & \alpha_0(-f_\rho g)\cdot q^{\frac{1}{24}+P_2(h_\rho g/\delta)/2}\cdot\sum\limits_{n=-\infty}^{\infty} (-1)^n\zeta_\delta^{-nf_\rho g}q^{\frac{n^2-n}{2}+\frac{n\tilde{h_\rho}}{\delta}}\Bigg]_\nu,
		\end{align*} 
		where $\tilde{h_\rho}$ is the representative in $[0,\delta)$ of $h_\rho g.$ 
		Noting that $|P_2(x)|\leq\frac{1}{6}$ for all $x,$ this is clearly the constant component of
		$$
		\Bigg[\alpha_0(-f_\rho g)^{-1}q^{-\frac{P_2(h_\rho g/\delta)}{2}-\frac{1}{24}},\alpha_0(-f_\rho g)q^{\frac{1}{24}+P_2(h_\rho g/\delta)/2}\Bigg]_\nu.
		$$
		By definition of the Rankin--Cohen bracket, this is given by
			$$
		\sum\limits_{\substack{r+s=\nu \\ r,s\geq 0}}(-1)^r\frac{\Gamma(\nu-1/2)\Gamma(\nu+1/2)}{s!r!\Gamma(\nu-1/2-s)\Gamma(\nu+1/2-r)}(-1)^r \left(\frac{P_2\left(\frac{h_\rho g}{\delta}\right)}{2}+\frac{1}{24}\right)^{r+s}.
		$$
		Using the definition of the classical hypergeometric function, we can rewrite this as
	$$
	\Gamma(\nu-\frac{1}{2})\Gamma(\nu+\frac{1}{2})\left(\frac{P_2\left(\frac{h_\rho g}{\delta}\right)}{2}+\frac{1}{24}\right)^\nu\cdot \frac{1}{\nu!\Gamma(-1/2)\Gamma(\nu+1/2)}{_2F_1}\left(\begin{array}{cc}
		-\nu & -\nu+\frac{1}{2}\\
		~& -\frac{1}{2} \\
	\end{array}\mid 1\right).
	$$
	Gauss's hypergeometric theorem (see (1.3) of \cite{bailey}) then gives us that the constant term is
	$$
	\frac{(2\nu-2)!}{\nu!(\nu-2)!}\cdot\left(\frac{P_2\left(\frac{h_\rho g}{\delta}\right)}{2}+\frac{1}{24}\right)^\nu.
	$$
	The theorem then follows immediately from Lemmas~\ref{lem:EisenSimplify} and \ref{lem:FourierBracket} and Proposition~\ref{prop:PeterssonInner}.
	\end{proof}
	
	\begin{proof}[Proof of Theorem~\ref{thm:Rademacher}]
	This follows from comparing Fourier coefficients in Proposition~\ref{prop:EtaPoincare}.	
	\end{proof}


\begin{thebibliography}{99}
		\bibitem{Andrews} G. E. Andrews, \emph{The theory of partitions}, Cambridge University Press, Cambridge, 1984.
		
		\bibitem{bailey}
		W. Bailey, {\it Generalized hypergeometric series}, Cambridge Tracts in Mathematics and Mathematical Physics 
		\textbf{32}, Cambridge Univ. Press, Cambridge, 1935.
		
		\bibitem{Bhowmik} T. Bhowmik, W. Tsai, D. Ye, \emph{Euler-type recurrences for $t$-color partitions and $t$-regular partition functions}, Res. Math. Sci. {\bf 12} (2025), no. 4, Paper No. 69.
		
		\bibitem{BringmannOno} K. Bringmann and K. Ono, \emph{Coefficients of harmonic Maass forms}, Proceedings of the 2008 University of Florida Conference on Partitions, $q$-series and Modular Forms. Dev. Math., vol. 23 (2012), 23--28.
		
		\bibitem{HMFBook} K. Bringmann, A. Folsom, K. Ono, and L. Rolen, \emph{Harmonic Maass forms and mock modular forms: Theory and applications,} Amer. Math. Soc. Colloq. {\bf 64}, Amer. Math. Soc., Providence 2017.
		
		\bibitem{Stromberg}
		H. Cohen and F. Str\"omberg, \emph{ Modular forms: A classical approach}, Graduate Studies in Mathematics, Vol \textbf{179}, Amer. Math. Soc., Providence, 2017.
		
		\bibitem{Gomez} K. Gomez, K. Ono, H. Saad, and A. Singh, \emph{Pentagonal number recurrence relations for $p(n)$}, Adv. Math. {\bf 474} (2025), Paper No. 110308.
		
		\bibitem{Rosen} K. Ireland and M. Rosen, \emph{A classical introduction to modern number theory}, Grad. Texts in Math. {\bf 84}, Springer-Verlag, New York, 1990.
		
		\bibitem{Kowalski} H. Iwaniec and E. Kowalski, \emph{Analytic number theory}, Amer. Math. Soc. Colloq. Publ. {\bf 53}, American Mathematical Society, Providence, RI, 2004.
		
		\bibitem{RankinBook} R. A. Rankin, \emph{Modular forms and functions}, Cambridge Univ. Press, Cambridge, 1977.
		
		\bibitem{Schoenberg} B. Schoeneberg, \emph{Elliptic modular functions}, Springer Berlin, Heidelberg, 1974.
		
		\bibitem{Wang} W. Wang, \emph{Recurrence and congruences for the smallest parts function}, \texttt{https://arxiv.org/abs/2512.10658}, preprint.

	\end{thebibliography}
\end{document}